\documentclass[final,onefignum,onetabnum]{siamart220329}

\usepackage{braket,amsfonts}

\usepackage{array}

\usepackage[caption=false]{subfig}

\usepackage{pgfplots}

\newsiamthm{claim}{Claim}
\newsiamremark{remark}{Remark}
\newsiamremark{hypothesis}{Hypothesis}
\crefname{hypothesis}{Hypothesis}{Hypotheses}

\usepackage{algorithmic}
\usepackage{graphicx,epstopdf}

\Crefname{ALC@unique}{Line}{Lines}

\usepackage{amsopn}

\usepackage{xspace}
\usepackage{bold-extra}
\usepackage[most]{tcolorbox}
\usepackage{cleveref}
\usepackage{mathtools}
\usepackage{pifont}
\newcommand{\xmark}{\ding{55}}%

\crefname{figure}{Fig.}{Figs.}
\Crefname{figure}{Fig.}{Figs.}

\colorlet{texcscolor}{blue!50!black}
\colorlet{texemcolor}{red!70!black}
\colorlet{texpreamble}{red!70!black}
\colorlet{codebackground}{black!25!white!25}

\usepackage{colortbl}

\newcount\Comments  
\newcommand{\kibitz}[2]{\ifnum\Comments=1\textcolor{#1}{#2}\fi}

\newcommand{\RZ}[1]    {\kibitz{teal}   {[RZ: #1]}}

\newcommand{\m}{}
\DeclareMathOperator{\Tr}{\mathrm{Tr}}
\DeclareMathOperator{\D}{\Delta}

\DeclareMathOperator{\St}{\mathrm{St}}

\newtheorem{thm}{Theorem}[section]
\newtheorem{lem}[thm]{Lemma}

\newtheorem{rem}[thm]{Remark}
\newtheorem{prob}[thm]{Problem}
\newtheorem{cond}[thm]{Condition}

\def \new {}
\def \newrev {}

\lstdefinestyle{siamlatex}{%
  style=tcblatex,
  texcsstyle=*\color{texcscolor},
  texcsstyle=[2]\color{texemcolor},
  keywordstyle=[2]\color{texemcolor},
  moretexcs={cref,Cref,maketitle,mathcal,text,headers,email,url},
}

\tcbset{%
  colframe=black!75!white!75,
  coltitle=white,
  colback=codebackground, 
  colbacklower=white, 
  fonttitle=\bfseries,
  arc=0pt,outer arc=0pt,
  top=1pt,bottom=1pt,left=1mm,right=1mm,middle=1mm,boxsep=1mm,
  leftrule=0.3mm,rightrule=0.3mm,toprule=0.3mm,bottomrule=0.3mm,
  listing options={style=siamlatex}
}

\newtcblisting[use counter=example]{example}[2][]{%
  title={Example~\thetcbcounter: #2},#1}

\newtcbinputlisting[use counter=example]{\examplefile}[3][]{%
  title={Example~\thetcbcounter: #2},listing file={#3},#1}

\DeclareTotalTCBox{\code}{ v O{} }
{ 
  fontupper=\ttfamily\color{black},
  nobeforeafter,
  tcbox raise base,
  colback=codebackground,colframe=white,
  top=0pt,bottom=0pt,left=0mm,right=0mm,
  leftrule=0pt,rightrule=0pt,toprule=0mm,bottomrule=0mm,
  boxsep=0.5mm,
  #2}{#1}

\patchcmd\newpage{\vfil}{}{}{}
\begin{tcbverbatimwrite}{tmp_\jobname_header.tex}
\title{An efficient algorithm for the Riemannian logarithm on the Stiefel manifold for a family of Riemannian metrics}

\author{Simon Mataigne\thanks{UCLouvain, ICTEAM, Louvain-la-Neuve, Belgium. Simon Mataigne is a Research Fellow of the Fonds de la Recherche Scientifique - FNRS. \email{simon.mataigne@uclouvain.be}}
\and Ralf Zimmermann \thanks{University of Southern Denmark, Department of Mathematics and Computer Science,
Odense, Denmark.
\email{zimmermann@imada.sdu.dk}}
\and Nina Miolane\thanks{UC Santa Barbara, Electrical and Computer Engineering, Santa Barbara, CA. Nina Miolane acknowledges funding from the NSF grant 2313150. \email{ninamiolane@ucsb.edu}}}

\headers{A method for the Riemannian logarithm on Stiefel}{Simon Mataigne, Ralf Zimmermann, Nina Miolane}
\end{tcbverbatimwrite}
\input{tmp_\jobname_header.tex}

\ifpdf
\hypersetup{ pdftitle={LogStiefelmetrics}}
\fi

\begin{document}
\maketitle


\begin{abstract}
Since the popularization of the Stiefel manifold for numerical applications in 1998 in a seminal paper from Edelman et al., it has been exhibited to be a key to solve many problems from optimization, statistics and machine learning. In 2021, H\"uper et al.~proposed a one-parameter family of Riemannian metrics on the Stiefel manifold, subsuming the well-known Euclidean and canonical metrics. Since then, several methods have been proposed to obtain a candidate for the Riemannian logarithm given any metric from the family.  Most of these methods are based on the shooting method or rely on optimization approaches. For the canonical metric, Zimmermann proposed in 2017 a particularly efficient method based on a pure matrix-algebraic approach. In this paper, we derive a generalization of this algorithm 
that works for the one-parameter family of Riemannian metrics. The algorithm is proposed in two versions, termed backward and forward, for which we prove that it conserves the local linear convergence previously exhibited in Zimmermann's algorithm for the canonical metric.
\end{abstract}

\section{Introduction} 
The field of statistics on manifolds has experienced significant growth in recent years, driven by a multitude of applications involving manifold-valued data | such as  orthogonal frames, subspaces, fixed-rank matrices, diffusion tensors 
or shapes | that need to undergo processes like denoising, resampling, extrapolation, compression, clustering, or classification (see, e.g., \cite{LeBrigant2018, Jung2012AnalysisSpheres, Kent2005UsingStructure, Pennec2006b}). Given a Riemannian manifold $\mathcal{M}$, both statistical and numerical applications often require the computation of distances between data points $U$ and $\widetilde{U} \in \mathcal{M}$~\cite{frechet1948}. Calculating this distance relies on an oracle that provides a minimal geodesic between the two points. A  minimal geodesic is a curve whose length corresponds to the Riemannian distance between $U$ and $\widetilde{U}$.  Most geometries lack a known closed-form expression for the minimal geodesic between any two points. Consequently, algorithms must be developed for these geometries to yield a candidate for the minimal geodesic. The Stiefel manifold falls into this category. It has gained popularity since the publication of the seminal paper \cite{EdelmanArias98} by Edelman et al. This manifold finds applications in various fields, including statistics \cite{ChakrabortyVemuri2019}, optimization \cite{EdelmanArias98}, and deep learning \cite{GaoVaryAblinAbsil2022,Huang_Liu_Lang_Yu_Wang_Li_2018}. 
The implementation of the Stiefel manifold is notably available in software packages such as \href{https://geomstats.github.io/}{\texttt{Geomstats}} \cite{JMLR:v21:19-027}, \href{https://www.manopt.org/}{\texttt{Manopt}} \cite{manopt}, \href{https://manoptjl.org/stable/}{\texttt{Manopt.jl}} \cite{Bergmann:2022}.

It has been demonstrated in \cite{ZimmermannRalf22} that the Stiefel manifold, equipped with any Riemannian metric from the one-parameter family introduced in \cite{Huper2021}, features at least one minimal geodesic between any two points. However, computing \emph{any} geodesic between two given points on the Stiefel manifold is a challenging task, known as the \emph{geodesic endpoint problem}. Several research efforts have proposed numerical methods to address this issue \cite{Nguyengeodesics,RentmeestersQ, sutti2023shooting, Zimmermann17, ZimmermannRalf22}. Unfortunately, none of these methods guarantees the computation of a minimal geodesic, referred to as the \emph{logarithm problem}.  

The one-parameter family of metrics introduced in \cite{Huper2021} includes the well-known Euclidean and canonical metrics \cite{EdelmanArias98}. In \cite{nguyen2022curvature}, it was observed that this family of metrics can be considered as metrics arising from a Cheeger deformation. Cheeger deformation metrics were first introduced in \cite{Cheeger1973} and have since been used in Riemannian geometry to construct metrics with special features, e.g., non-negative curvature, or Einstein metrics. We refer to \cite{nguyen2022curvature} for further details and additional references.
Numerical algorithms for computing geodesics and Riemannian normal coordinates under this family of metrics were introduced in \cite{ZimmermannRalf22} and \cite{Nguyengeodesics}. 
The case of the canonical metric allowed for special treatment, which was exploited in \cite[Algorithm~4]{ZimmermannRalf22} and earlier in \cite{Zimmermann17}. The associated algorithm has proven local linear convergence. In the comparative study of \cite{ZimmermannRalf22}, it turned to be the most efficient and robust numerical approach. It was therefore a shortcoming that \cite[Algorithm~4]{ZimmermannRalf22} was only designed to work with the canonical metric.

\paragraph{Original contribution}
In this paper, we generalize \cite[Algorithm~4]{ZimmermannRalf22} to the family of metrics introduced in \cite{Huper2021}. The algorithm is provided in two versions, termed \emph{backward} and \emph{forward}. For the backward iteration, we show the local linear convergence of the algorithm and we obtain an explicit expression for the convergence rate, generalizing the one obtained in \cite[Proposition~7]{ZimmermannRalf22}. However, performing backward iterations require solving a nonlinear matrix equation. To alleviate the computational costs, we introduce three types of \emph{forward} iterations that avoid solving this nonlinear matrix equation. These forward iterations preserve the local linear convergence of the algorithm. 
\paragraph{Reproducibility statement} All the codes written to produce results and figures are available at the address \url{https://github.com/smataigne/StiefelLog.jl}.
\paragraph{Organization}
The paper is structured as follows. We start by recalling the necessary background on the Stiefel manifold in~\cref{sec:preliminaries}. In~\cref{sec:Zimmermannsalgorithm}, we recap
\cite[Algorithm~4]{ZimmermannRalf22} and its convergence result. We generalize the approach to the family of metrics in~\cref{sec:generalizedalgorithm}, where we provide a convergence analysis for the backward and forward iterations. Then, we compare the performance of the new algorithms with each other in~\cref{sec:forwardperformance}. Finally, the convergence radius  and the benchmark with methods from the state of the art is carried out in~\cref{sec:performance}.

\paragraph{Notations} For $p>0$, we denote the $p\times p$ identity matrix by $\m{I}_p$. For $n>p>0$, we define \begin{equation*}
    \m{I}_{n\times p }\coloneq\begin{bmatrix}
        \m{I}_p\\
        \m{0}_{(n-p)\times p}
    \end{bmatrix}\text{ and } \m{I}_{p\times n} \coloneq \m{I}_{n\times p }^T.
\end{equation*}
$\mathrm{Skew}(p)$ is the set of $p\times p$ skew-symmetric matrices ($\m{A}=-\m{A}^T)$ and the orthogonal group of $p\times p$ matrices is written as \begin{equation*}
    \mathrm{O}(p) \coloneq\{\m{Q}\in\mathbb{R}^{p\times p}\ | \ \m{Q}^T\m{Q}=\m{I}_p\}.
\end{equation*}
The special orthogonal group $\mathrm{SO}(p)$ is the subset of matrices of $ \mathrm{O}(p)$ with positive unit determinant. An orthogonal completion $\m{U}_\perp$ of $\m{U}\in\mathbb{R}^{n\times p}$ is $\m{U}_\perp\in\mathbb{R}^{n\times (n-p)}$ such that $\begin{bmatrix}
    \m{U}&\m{U}_\perp
\end{bmatrix}\in \mathrm{O}(n)$. Throughout this paper, $\exp$ and $\log$ always denote the matrix exponential and the principal matrix logarithm. The Riemannian exponential and logarithm are written with capitals, $\text{Exp}$ and $\text{Log}$. Finally, $\|\cdot\|_2$ denotes the spectral matrix norm.

\section{Background on the Stiefel manifold}\label{sec:preliminaries}
 The main references for this section are \cite{AbsMahSep2008,EdelmanArias98,ZimmermannRalf22}. Given integers $n \geq p > 0$, the Stiefel manifold of orthonormal $p$-frames in $\mathbb{R}^n$ is defined as
\begin{equation}
    \mathrm{St}(n,p)\coloneq\{ \m{U}\in \mathbb{R}^{n\times p}\  |\ \m{U}^T\m{U} = \m{I}_p, \ n\geq p \}.
\end{equation}
If $n=p$, the Stiefel manifold becomes the orthogonal group, $\mathrm{St}(n,n)=\mathrm{O}(n)$. This case allows for special treatment and is extensively studied. In particular, the minimal geodesics are known in closed form. Therefore, we restrict our considerations to $n>p$. The Stiefel manifold is a differentiable manifold of dimension $np-\frac{p(p+1)}{2}$~\cite{AbsMahSep2008}. The tangent space of $\mathrm{St}(n,p)$ at a point $\m{U}$ can be written as
\begin{align*}
    T_\m{U} \mathrm{St}(n,p) &=\{\m{\Delta}\in\mathbb{R}^{n\times p}\ | \ \m{U}^T\m{\Delta}+\m{\Delta}^T\m{U}=\m{0} \}.
\end{align*}
It follows that for all $\m{\Delta}\in T_\m{U} \mathrm{St}(n,p)$, we can write $\m{\Delta}=\m{UA}  + \m{U}_\perp \m{B}$ where $\m{A}\in \mathrm{Skew}(p)$, $\m{B} \in \mathbb{R}^{(n-p)\times p}$ and $\m{U}_\perp\in \mathrm{St}(n,n-p)$ is an orthogonal complement of $\m{U}$. The concept of tangent space is illustrated in~\cref{fig:exp_log}. Notice that for a given $\m{\Delta}\in T_\m{U} \mathrm{St}(n,p)$, $\m{A}$ is uniquely defined while $\m{B}$ depends on the chosen completion $\m{U}_\perp$. Indeed $\m{U}_\perp \m{B}=(\m{U}_\perp \m{R})(\m{R}^T \m{B})$ for all $\m{R}\in{\new \mathrm{O}(n-p)}$.
(For the initiated reader, we mention that selecting a specific completion $\m{U}_\perp$ corresponds to lifting the tangent vector $\Delta$ to a specific horizontal space. We omit the details.)
\subsection{Metrics and distances} 
Let $\mathcal{M}$ be a differentiable manifold. A Riemannian metric on $\mathcal{M}$ is a family {\new $\{\langle\cdot,\cdot\rangle^x:\mathrm{T}_x\mathcal{M}\times\mathrm{T}_x\mathcal{M}\mapsto\mathbb{R}\}_{x\in\mathcal{M}}$ of symmetric positive definite bi-linear forms that depends smoothly on the location $x\in\mathcal{M}$}.
In practice, the input arguments of the metric encode its dependency on $x$, so that we can write $\langle\cdot,\cdot\rangle$ without ambiguity. Edelman et al.~\cite{EdelmanArias98} introduced two natural metrics on the Stiefel manifold, called the \emph{Euclidean} and the \emph{canonical} metric, respectively. These two metrics are subsumed by the family of metrics introduced in~\cite{Huper2021}. For convenience, we use another parameterization of this family  and define the $\beta$-metric\footnote{\cite{absil2024ultimate} and \cite{Nguyengeodesics} also used this more convenient parameterization.} with $\beta>0$ as follows. For all $\m{\Delta},\widetilde{\m{\Delta}}\in\mathrm{T}_\m{U}\mathrm{St}(n,p)$, 
\begin{equation*}
    \langle\m{\Delta},\widetilde{\m{\Delta}}\rangle_\beta \coloneq \Tr \m{\Delta}^T(\m{I}_n-(1-\beta)\m{UU}^T)\widetilde{\m{\Delta}}=\beta \Tr \m{A}^T\widetilde{\m{A}} +\Tr \m{B}^T\widetilde{\m{B}}.
\end{equation*}
The Euclidean and canonical metrics correspond to $\beta = 1$ and $\beta = \frac{1}{2}$ respectively. Therefore, we are particularly interested in $\beta\in\left[\frac{1}{2},1\right]$ and our experiments will focus on this interval. The Euclidean metric is inherited from the ambient Euclidean space $\mathbb{R}^{n\times p}$ while the canonical metric is inherited from the quotient structure $\mathrm{St}(n,p)=\mathrm{SO}(n)/\mathrm{SO}(n-p)$~\cite{EdelmanArias98}. The norm induced by any $\beta$-metric is $\|\m{\Delta} \|_\beta= \sqrt{\langle \m{\Delta},\m{\Delta}\rangle_\beta}$, and the length of a continuously differentiable curve $\gamma:[0,1]\mapsto\mathrm{St}(n,p)$ is given by 
\begin{equation}\label{eq:length}
    l_\beta(\gamma) = \int_0^1 \|\dot{\gamma}(t)\|_\beta \ \mathrm{d}t,
\end{equation}
where $\dot{\gamma}$ denotes the time derivative of $\gamma$. For all $\m{U},\widetilde{\m{U}}\in \mathrm{St}(n,p)$, we obtain an induced distance function~\cite[Proposition~1.1]{sakai1996riemannian} (also termed \emph{Riemannian distance}) defined as 
\begin{equation}\label{eq:geodesicdistance}
     d_\beta(\m{U},\widetilde{\m{U}})=\inf\{l_\beta(\gamma)\ |\  \gamma(0)=\m{U},\ \gamma(1) = \widetilde{\m{U}}\}.
 \end{equation}
 Since $\beta$-geodesics exist for all times \cite{Huper2021}, the Hopf-Rinow theorem ensures that a curve achieves the minimal distance: the ``$\inf$'' in~\eqref{eq:geodesicdistance} is a ``$\min$'' when $\mathcal{M}=\mathrm{St}(n,p)$.
 \begin{figure}
     \centering
     \includegraphics[width = 7cm]{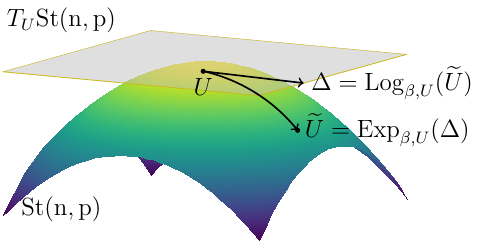}
     \vspace{-0.2cm}
     \caption{Conceptual illustration of the Stiefel manifold $\mathrm{St}(n,p)$, the tangent space $T_\m{U}\mathrm{St}(n,p)$, the exponential map $\mathrm{Exp}_{\beta,\m{U}}(\m{\Delta})$ and the logarithmic map $\mathrm{Log}_{\beta,\m{U}}(\widetilde{U})$ (see Section~\cref{subsec:exponential}).}
     \label{fig:exp_log}
     \vspace{-0.5cm}
 \end{figure}
 
 Riemannian metrics act on tangent vectors, {\new  while} the Riemannian distance is for points on the manifold. Since $\mathrm{St}(n,p)$ is embedded in $\mathbb{R}^{n\times p}$, the distance between $U,\tilde U\in\mathrm{St}(n,p)$ can also be measured in the ambient space. An upper bound for this distance in the Frobenius norm is
 \[
    \|U-\tilde U\|_F \leq \|U\|_F + \|\tilde U\|_F = 2\sqrt{\mathrm{Tr}(I_p)} = 2\sqrt{p},
 \]
 and is referred to as the {\em Frobenius diameter} of the Stiefel manifold.
 \subsection{The Riemannian exponential and logarithm}\label{subsec:exponential}  Consider $\mathrm{St}(n,p)$ endowed with $\langle\cdot,\cdot\rangle_\beta$, written $\left(\mathrm{St}(n,p),\langle\cdot,\cdot\rangle_\beta\right)$. Zimmermann and Hüper~\cite{ZimmermannRalf22} showed that the \emph{Riemannian exponential} $\mathrm{Exp}_{\beta,\m{U}}:T_\m{U}\mathrm{St}(n,p)\mapsto\mathrm{St}(n,p)$, i.e., the function that maps $\m{\Delta}\in T_\m{U}\mathrm{St}(n,p)$ to the point reached at unit time by the geodesic with starting point $\m{U}$ and initial velocity $\m{\Delta}$~(see, e.g., \cite[Chap.~II,~Sec.~2]{sakai1996riemannian}),  can be expressed according to~\cref{thm:geodesics}. The mode of operation of the Riemannian exponential is illustrated in~\cref{fig:exp_log}.
 \begin{thm}{\textsc{The Riemannian exponential} \cite[Equation~10]{ZimmermannRalf22}.}\label{thm:geodesics}
    \it
    For all $\m{U}\in \mathrm{St}(n,p)$ and $\m{\Delta} \in T_\m{U}\mathrm{St}(n,p)$, we have
    \begin{equation}\label{eq:paramgeogeneral}
        \mathrm{Exp}_{\beta,\m{U}}(\m{\Delta}) = \begin{bmatrix}
            \m{U} & \ \m{Q}
        \end{bmatrix}\exp\left(\begin{bmatrix}
            2\beta \m{A}& -\m{B}^T\\
            \m{B} & \m{0}
        \end{bmatrix}\right)\m{I}_{n\times p}\exp\left((1-2\beta)\m{A}\right),
    \end{equation}
    where $\m{A} = \m{U}^T \m{\Delta}\in \mathrm{Skew}(p)$ and $\m{Q}\m{B} = (\m{I}-\m{U}\m{U}^T)\m{\Delta}\in\mathbb{R}^{n\times p}$ is any matrix decomposition where $\m{Q}\in\mathrm{St}(n,n-p)$ with $\m{Q}^T \m{U}=\m{0}$ and $\m{B}\in\mathbb{R}^{(n-p)\times p}$.
\end{thm}
If $n>2p$, we can always reduce the decomposition of $\m{\Delta}$ such that $\m{Q}\in\mathrm{St}(n,p)$ and $\m{B}\in\mathbb{R}^{p\times p}$~\cite{RentmeestersQ, ZimmermannRalf22}.
The first matrix exponential of \eqref{eq:paramgeogeneral} belongs then to $\mathbb{R}^{2p\times 2p}$ instead of $\mathbb{R}^{n\times n}$. {\newrev This property yields significant computational savings if $n\gg p$.} Given $U, \widetilde{U}\in \mathrm{St}(n,p)$, the Riemannian logarithm $\mathrm{Log}_{\beta,\m{U}}(\widetilde{\m{U}})$ is the function returning the set of all minimal-norm tangent vectors $\Delta\in T_U\mathrm{St}(n,p)$ such that $\mathrm{Exp}_{\beta,\m{U}}(\m{\Delta})=\widetilde{U}$. Locally, the Riemannian logarithm is the inverse of the exponential.
    \begin{prob}{\textsc{The Logarithm problem}.}\label{prob:truelogproblem}
    \it
    Let $\m{U},\m{\widetilde{U}}\in \mathrm{St}(n,p)$ and $\beta>0$. The Riemannian logarithm $\mathrm{Log}_{\beta,\m{U}}(\widetilde{\m{U}})$ returns all $\m{\Delta}\in T_{\m{U}}\mathrm{St}(n,p)$ such that
    \begin{equation*}
         \mathrm{Exp}_{\beta,\m{U}}(\m{\Delta})=\widetilde{\m{U}} \text{ and } \|\m{\Delta}\|_\beta = d_\beta(\m{U},\widetilde{\m{U}}).
    \end{equation*}
    The curves $[0,1]\ni t\mapsto \mathrm{Exp}_{\beta,\m{U}}(t\m{\Delta})$ are then called \emph{minimal geodesics}.
\end{prob}
Because the Stiefel manifold is complete, the Hopf-Rinow theorem \cite[Chap.~7, Thm.~2.8]{DoCarmo2013riemannian} ensures the existence of at least one minimal geodesic between any two points on the Stiefel manifold. However, on compact Riemannian manifolds, every geodesic has a designated \emph{cut time} $t^*$ \cite{sakai1996riemannian}: the geodesic stays minimal as long as $t \leq t^*$, not beyond. The point $\mathrm{Exp}_{\beta,\m{U}}(t^*\m{\Delta})$ is termed the \emph{cut point}. Consequently, when a geodesic between two points is established, the shooting direction is a Riemannian logarithm if and only if the destination is reached before the cut point. 
In \cite[Thm.~3]{ZimmermannRalf22}, it is emphasized that the geodesic endpoint problem on the Stiefel manifold boils down to solving a nonlinear matrix problem. Very recent works \cite{absil2024ultimate, stoye2024injectivity} proposed a thin interval for the injectivity  radius $\mathrm{inj}_{\mathrm{St}(n,p)}$, i.e., the distance to the nearest cut point. In particular for $\beta =\frac{1}{2}$, they showed that $0.894\pi<\mathrm{inj}_{\mathrm{St}(n,p)}<0.914\pi$. Within the injectivity radius,  any solution to the geodesic endpoint problem gives the unique solution to~\cref{prob:truelogproblem}, namely, \emph{the} Riemannian logarithm.

\section{State of the art}\label{sec:Zimmermannsalgorithm}  \cite{RentmeestersQ,Zimmermann17}~introduced an efficient method to compute a geodesic between any two given points $\m{U},\m{\widetilde{U}}\in \mathrm{St}(n,p)$, but only in the case of the canonical metric ($\beta=\frac{1}{2}$) and $n\geq2p$. The algorithm was enhanced in \cite{ZimmermannRalf22} to obtain a better convergence rate. Its pseudo-code is given by~\cref{alg:Zimmermannsalgorithm}. Previous experiments~\cite[Table~1]{ZimmermannRalf22} highlighted the better performance of~\cref{alg:Zimmermannsalgorithm} compared to the shooting method, introduced first in \cite{BrynerDarshan17}. For the $\beta$-metric, \cite{Nguyengeodesics} also proposed a L-BFGS method inspired from the shooting method. The L-BFGS method was however shown to converge slower than the shooting method \cite[Table~2]{Nguyengeodesics}. This observation provides a strong incentive to generalize~\cref{alg:Zimmermannsalgorithm} to the parameterized family of $\beta$-metrics. 

\cref{alg:Zimmermannsalgorithm} finds a vector $\Delta$ satisfying $\widetilde{U} = \mathrm{Exp}_{\beta,U}(\Delta)$ when $\beta=\frac{1}{2}$. In this case,~\cref{prob:truelogproblem} simplifies and consists of finding matrices $\m{A}\in\mathrm{Skew}(p)$ and $\m{B}\in \mathbb{R}^{p \times p}$ such that 
\begin{equation}\label{eq:nonlinearproblem}
    \begin{bmatrix}
        \m{M}\\
        \m{N}
    \end{bmatrix} = \exp\begin{bmatrix}
        \m{A}&-\m{B}^T\\
        \m{B}&\m{0}
    \end{bmatrix}\m{I}_{2p\times p},
\end{equation}
where $\m{M}:=\m{U}^T\widetilde{U}$, $\m{N}:=\m{Q}^T\widetilde{U}$ with $\m{Q}\in\mathrm{St}(n,p)$ {\new such that $\mathrm{col}([U\ \widetilde{U}])\subseteq\mathrm{col}([U\ Q])$} and $\m{Q}^T\m{U}=\m{0}$. Equation \eqref{eq:nonlinearproblem} is solved iteratively by computing a sequence of matrices $\{\m{A}_k,\m{B}_k,\m{C}_k\}_{k\in\mathbb{N}}$ by
\begin{equation}\label{eq:sequenceforZimmermann}
    \begin{bmatrix}
        \m{A}_k&-\m{B}_k^T\\
        \m{B}_k&\m{C}_k
    \end{bmatrix} \coloneq\log \m{V}_k\in\mathrm{Skew}(2p)\text{ with } \m{V}_k\coloneq\begin{bmatrix}
        \m{M}&\m{O}_k\\
        \m{N}&\m{P}_k
    \end{bmatrix}\in \mathrm{SO}(2p).
\end{equation}
$O_0,P_0$ are chosen such that $V_0\in\mathrm{SO}(2p)$ and $\m{V}_k$ is updated to have $\lim_{k\rightarrow\infty} \|\m{C}_k\|_2 = 0$. \cref{thm:convergenceZimmermann} below shows that taking $\m{V}_{k+1} = \m{V}_k\begin{bmatrix}
    \m{I}_p&\m{0}\\
    \m{0}&\exp(\m{\Gamma}_k)
\end{bmatrix}$, where $\m{\Gamma}_k$ solves the \emph{Sylvester equation}
\begin{equation}\label{eq:sylvesterfirst}
    \m{S\Gamma}_k+\m{\Gamma}_k\m{S} = \m{C}_k \text{ with } \m{S}\coloneq\frac{1}{12}\m{B}_k\m{B}_k^T-\frac{\m{I}_p}{2},
\end{equation} yields local linear convergence, i.e., $\|\m{C}_{k+1}\|_2 \leq a \|\m{C}_k\|_2+\mathcal{O}(\|\m{C}_k\|_2^2)$ with $a<1$. Ultimately, \cref{alg:Zimmermannsalgorithm} outputs a geodesic $\gamma:[0,1]\mapsto\St(n,p)$ between $\m{U}$ and $\widetilde{U}$ defined by
\begin{equation}
    \gamma(t)\coloneq\mathrm{Exp}_{\frac{1}{2}, \m{U}}(t\m{\Delta})= \begin{bmatrix}
        \m{U}&\m{Q}
    \end{bmatrix}\exp \left(t\begin{bmatrix}
        \m{A}_\infty&-\m{B}^T_\infty\\
        \m{B}_\infty&\m{0}
    \end{bmatrix}\right)\m{I}_{2p \times p}.
\end{equation}
\cref{alg:Zimmermannsalgorithm} possesses two strengths. First, $\{\m{A}_k,\m{B}_k,\m{C}_k\}$ provides a curve between $\m{U}$ and $\widetilde{U}$ at any iteration $k$, which becomes a (numerical) geodesic at convergence when $\|\m{C}_k\|_2<\varepsilon$ for some tolerance $\varepsilon>0$. Second, unlike the shooting method, no time-discretization of the geodesic is needed. This is a strong advantage since i) the discretization is computationally expensive: a Riemannian exponential and an approximate parallel transport have to be computed at every discretized point along the geodesic and ii) the level of discretization is a user-parameter that has to be tuned, or guessed. 
On the other hand, the shooting method involves computing the matrix exponential of skew-symmetric matrices. At this day, it can be done more efficiently | such as in \href{https://github.com/JuliaLinearAlgebra/SkewLinearAlgebra.jl}{\texttt{SkewLinearAlgebra.jl}} | than computing the matrix logarithm in \eqref{eq:sequenceforZimmermann}. Nonetheless \cite[Table~1]{ZimmermannRalf22} observed that~\cref{alg:Zimmermannsalgorithm} runs faster due to its better convergence rate. It was thus a deficiency that the most efficient method among the considered ones could not be generalized to the $\beta$-metrics. We close this gap in the next sections.
\begin{algorithm}
     \caption{\cite[Algorithm~4]{ZimmermannRalf22} Improved algebraic Stiefel logarithm for the canonical metric ($\beta=\frac{1}{2}$).} \label{alg:Zimmermannsalgorithm}
    \begin{algorithmic}[1]
        \STATE \textbf{INPUT:} Given $\m{U},\widetilde{U}\in\St(n,p)$, $n\geq2p$ and $\varepsilon>0$, compute:
        \STATE Set $\m{M} = \m{U}^T\widetilde{U}$.
        \STATE Compute $\widehat{Q}\widehat{N} = (\m{I} - \m{U}\m{U}^T)\widetilde{U}$.\hspace{2.2cm} \#where $\widehat{Q}\in\mathrm{St}(n,p)$ and $\widehat{Q}^T U = 0$.
        \STATE Build $\m{V}_0 = \begin{bmatrix}
            \m{M}&\m{O}_0\\
            \m{N}&\m{P}_0
        \end{bmatrix}\in\mathrm{SO}(2p),\  Q\in\mathrm{St}(n,p)$. \hspace{-0.08cm} \#see~\cref{app:initialization} for $\m{N,O_0,P_0, Q}$.
        {\FOR{$k = 0, 1, ...$}
            \STATE Compute $\begin{bmatrix}
                \m{A}_{k} & -\m{B}_{k}^T\\
                \m{B}_{k}& \m{C}_{k}
            \end{bmatrix} = \log\m{V}_k$.
            \IF{$\|\m{C}_{k}\|\leq \varepsilon$}
                \STATE Break.
            \ENDIF
            \STATE Compute $\m{\Gamma}_k$ solving $\m{\Gamma}_k\m{S}+\m{S\Gamma}_k = \m{C}_k$ where $\m{S}\coloneq\frac{1}{12}\m{B}_k\m{B}_k^T-\frac{\m{I}_p}{2}$.
            \STATE Update $\m{V}_{k+1} = \m{V}_k\begin{bmatrix}
                \m{I}_p&0\\
                0&\exp(\m{\Gamma}_k)
            \end{bmatrix}$.
        \ENDFOR}
        \RETURN {\new $\m{\Delta} = \m{U}\m{A}_k+\m{QB}_k\in T_U \mathrm{St}(n,p)$.}
    \end{algorithmic}
\end{algorithm}
\begin{rem}\label{rem:rankdefficiency}
    In line 3 of~\cref{alg:Zimmermannsalgorithm}, \cite{Zimmermann17} proposed to compute a thin QR decomposition ${\new \widehat{Q}\widehat{N}} = (\m{I} - \m{U}\m{U}^T)\widetilde{U}$. It is without ambiguity if $(\m{I} - \m{U}\m{U}^T)\widetilde{U}$ has full column rank. However, if $(\m{I} - \m{U}\m{U}^T)\widetilde{U}$ is not full column rank, it is important to ensure ${\new \widehat{Q}}\in\mathrm{St}(n,p)$ with ${\new \widehat{Q}}^T\m{U}=\m{0}$.  Otherwise, \cref{alg:Zimmermannsalgorithm} may fail to find any $\m{\Delta}\in \mathrm{Log}_{\beta,U}(\widetilde{U})$. 
    Indeed, there are provable cases where
    \begin{equation}\label{eq:rankdef}
        \mathrm{rank}((\m{I} - \m{U}\m{U}^T)\Delta)>\mathrm{rank}((\m{I} - \m{U}\m{U}^T)\widetilde{U}). 
    \end{equation}
    We show in~\cref{app:rankdeficiency} that~\eqref{eq:rankdef} can only happen if $\widetilde{U}$ belongs to the cut locus of~$U$.
\end{rem}
The linear convergence result for~\cref{alg:Zimmermannsalgorithm} is given by~\cref{thm:convergenceZimmermann}. In~\cref{sec:convergencebackward}, 
we first generalize the algorithm to work with the parametric family of metrics in form of~\cref{alg:generalizedStiefel}.
The associated generalization of the convergence statement~\cref{thm:convergenceZimmermann} is then~\cref{thm:convergence}, which reduces to the original result for $\beta=\frac{1}{2}$.
\begin{thm}{\cite[Prop.~7]{ZimmermannRalf22}}\label{thm:convergenceZimmermann}
    Given $\m{U},\widetilde{U}\in\mathrm{St}(n,p)$ $(\m{U}\neq\widetilde{U})$ such that~\cref{alg:Zimmermannsalgorithm} converges. Assume further that there is $0<\delta<1$ such that for $\log(\m{V}_k)=
    \begin{bmatrix}
        \m{A}_{k} & -\m{B}_{k}^T\\
        \m{B}_{k}& \m{C}_{k}
    \end{bmatrix}$, it holds $\|\log(\m{V}_k)\|_2<\delta$ throughout the algorithm's iteration loop. Then, for $k$ large enough, it holds
    \begin{equation}
        \|\m{C}_{k+1}\|_2\leq\frac{6}{6-\delta^2}\frac{\delta^4}{1-\delta}\|\m{C}_{k}\|_2+\mathcal{O}(\|\m{C}_{k}\|_2^2).
    \end{equation}
\end{thm}
\cref{thm:convergenceZimmermann} highlights that the closer $\m{U}$ and $\widetilde{U}$, the faster~\cref{alg:Zimmermannsalgorithm} should converge. This property was numerically confirmed in \cite[Sec.~5.3]{Zimmermann17}.

\section{The generalization of~\cref{alg:Zimmermannsalgorithm}}\label{sec:generalizedalgorithm}
We propose a new approach to generalize~\cref{alg:Zimmermannsalgorithm} to all the $\beta$-metrics. Still, we assume $n\geq2p$ which matches the `$n\gg p$' setting of practical big-data applications. 

When $\beta\neq\frac{1}{2}$, a generalization of~\cref{alg:Zimmermannsalgorithm}  consists of finding
a sequence of matrices $\{\m{A}_k,\m{B}_k,\m{C}_k\}_{k\in\mathbb{N}}$ with  $\m{A}_k,C_k \in\mathrm{Skew}(p),\m{B}_k\in \mathbb{R}^{p \times p}$ such that
\begin{equation}\label{eq:sequenceforbeta}
    \begin{bmatrix}
        \m{M}\\
        \m{N}
    \end{bmatrix} = \exp\begin{bmatrix}
        2\beta\m{A}_k&-\m{B}_k^T\\
        \m{B}_k&\m{C}_k
    \end{bmatrix}\m{I}_{2p\times p}\exp((1-2\beta)\m{A}_k),
\end{equation}
where $\m{M}\coloneq\m{U}^T\widetilde{U}$, $\m{N}\coloneq\m{Q}^T\widetilde{U}$ ($\m{Q}\in\mathrm{St}(n,p)$, $\m{Q}^T\m{U}=\m{0}$) 
and $\lim_{k\rightarrow\infty} \|\m{C}_k\|_2 = 0$. Since $A_k$ appears twice in~\eqref{eq:sequenceforbeta}, it is a challenging non-linear equation to solve. Moreover, straightforward usage of Newton's method is very expensive \cite{ZimmermannRalf22}. However, the equation can be tackled by constructing a sequence
\begin{equation}\label{eq:tacklingdifficulty}
    \begin{bmatrix}
        2\beta\m{A}_k&-\m{B}_k^T\\
        \m{B}_k&\m{C}_k
    \end{bmatrix}\coloneq \log\left(\m{V}_k\begin{bmatrix}
        \exp(-(1-2\beta)\widehat{\m{A}}_k)&0\\
        0&\m{I}_p
    \end{bmatrix}\right),
\end{equation}
where $\widehat{\m{A}}_k$ is a consistent approximation, i.e., $\lim_{k\rightarrow\infty}\|\widehat{\m{A}}_k-\m{A}_k\|_2 =0$. $V_k$ follows its definition from~\eqref{eq:sequenceforZimmermann}.
The matrix sequence $\{\m{V}_k\}_{k\in\mathbb{N}}$ is chosen such that $\lim_{k\rightarrow\infty} \|\m{C}_k\|_2 = 0$. If one knows $\widehat{\m{A}}_k=\m{A}_k$, we term  it a \emph{backward iteration}. This is too expensive in practice and we rather perform a \emph{forward iteration} by selecting $\widehat{\m{A}}_k$ as an approximation of $\m{A}_k$. Many different ideas can be exploited, we propose three of them. The simplest one is $\widehat{\m{A}}_k\coloneq\m{A}_{k-1}$. We term it a \emph{fixed forward iteration}. To improve this approximation, we can solve~\eqref{eq:sequenceforbeta} for $A_k$ approximately and at low computational cost, based on our knowledge of $\m{A}_{k-1}$. We term this a \emph{pseudo-backward iteration}. Finally, we can improve the fixed forward approximation by taking $\widehat{\m{A}}_k\coloneq \m{A}_{k-1}+h\m{Q}_{k-1} (\m{A}_{k-1}-\widehat{\m{A}}_{k-1})\m{Q}_{k-1}^T$ for some $h\in\mathbb{R}$ and $\m{Q}_{k-1}\in\mathrm{O}(p)$. We term this an \emph{accelerated forward iteration}. All three possibilities are 
investigated in~\cref{sec:forwardconvergence}. \cref{alg:generalizedStiefel} is a pseudo-code of the generalization of~\cref{alg:Zimmermannsalgorithm} to the $\beta$-metric. 
\begin{algorithm}
     \caption{Improved algebraic Stiefel logarithm for the $\beta$-metric family.} \label{alg:generalizedStiefel}
    \begin{algorithmic}[1]
        \STATE \textbf{INPUT:} Given $\m{U},\widetilde{U}\in\mathrm{St}(n,p)$, $n\geq 2p$ and $\varepsilon>0$, compute:
        \STATE Set $\m{M} = \m{U}^T\widetilde{U}$.
        \STATE Compute $\widehat{Q}\widehat{N} = (\m{I} - \m{U}\m{U}^T)\widetilde{U}$.\hspace{2.2cm} \#where $\widehat{Q}\in\mathrm{St}(n,p)$ and $\widehat{Q}^T U = 0$.
        \STATE Build $\m{V}_0 = \begin{bmatrix}
            \m{M}&\m{O}_0\\
            \m{N}&\m{P}_0
        \end{bmatrix}\in\mathrm{SO}(2p),\  Q\in\mathrm{St}(n,p)$. \hspace{-0.08cm} \#see~\cref{app:initialization} for $\m{N,O_0,P_0, Q}$.
        {\FOR{$k = 0,1, ...$}
            \STATE Take an approximation $\widehat{\m{A}}_{k}\approx\m{A}_{k}$. \hspace{3.6cm} \#see~\cref{sec:forwardconvergence}.
            \STATE Compute $\begin{bmatrix}
                2\beta \m{A}_{k} & -\m{B}_{k}^T\\
                \m{B}_{k}& \m{C}_{k}
            \end{bmatrix} = {\new \log}\left(\m{V}_k\begin{bmatrix}
                \exp(-(1-2\beta)\widehat{\m{A}}_{k})&0\\
                0&\m{I}_p
            \end{bmatrix}\right)$.
            \IF{$\|\m{C}_{k}\|+\|\widehat{\m{A}}_{k}-\m{A}_{k}\|\leq \varepsilon$}
                \STATE Break.
            \ENDIF
            \STATE Compute $\m{\Gamma}_k$ solving $\m{\Gamma}_k\m{S}+\m{S\Gamma}_k = \m{C}_k$ where $\m{S}:=\frac{1}{12}\m{B}_k\m{B}_k^T-\frac{\m{I}_p}{2}$.
            \STATE Update $\m{V}_{k+1} = \m{V}_k\begin{bmatrix}
                \m{I}_p&0\\
                0&\exp(\m{\Gamma}_k)
            \end{bmatrix}$.
        \ENDFOR}
        \RETURN {\new $\m{\Delta} = \m{U}\m{A}_k+\m{QB}_k\in T_U \mathrm{St}(n,p)$.}
    \end{algorithmic}
\end{algorithm}
\begin{rem}
    At line 3 of~\cref{alg:generalizedStiefel}, \cref{rem:rankdefficiency} still holds.
\end{rem}
\subsection{The fundamental equation of~\cref{alg:generalizedStiefel}} We write explicitly how an iteration of the algorithm relates to the previous one. This will facilitate our analysis of the algorithm.
{\new By line 7 of~\cref{alg:generalizedStiefel}, it holds that
\begin{equation*}
	V_k = \exp\left(\begin{bmatrix}
            2\beta \m{A}_{k} & -\m{B}_{k}^T\\
            \m{B}_{k}& \m{C}_{k}
        \end{bmatrix}\right)\exp\left(\begin{bmatrix}
           (1-2\beta)\widehat{A}_k&0\\
            0&0
        \end{bmatrix}\right).
\end{equation*}
Introducing line 12 of~\cref{alg:generalizedStiefel}, i.e., $\m{V}_{k+1} = \m{V}_k\left[\begin{smallmatrix}
                \m{I}_p&0\\
                0&\exp(\m{\Gamma}_k)
            \end{smallmatrix}\right]$, we obtain
}
\begin{equation}\label{eq:fundamentalequation}
        \begin{bmatrix}
            2\beta \m{A}_{k+1} & -\m{B}_{k+1}^T\\
            \m{B}_{k+1}& \m{C}_{k+1}
        \end{bmatrix} = \log\left(\exp\left(\begin{bmatrix}
            2\beta \m{A}_{k} & -\m{B}_{k}^T\\
            \m{B}_{k}& \m{C}_{k}
        \end{bmatrix}\right)\exp\left(\begin{bmatrix}
            \m{\Theta}_k&0\\
            0&\m{\Gamma}_k
        \end{bmatrix}\right)\right),
\end{equation}
where $\m{\Theta}_k$ is defined by
\begin{equation}\label{eq:thetadefinition}
    \m{\Theta}_k \coloneq\log\left(\exp((1-2\beta)\widehat{\m{A}}_{k}) \exp(-(1-2\beta)\widehat{\m{A}}_{k+1})\right).
\end{equation}
The loop generated by~\eqref{eq:fundamentalequation} is driven by two parameters: $\m{\Gamma}_k$ and $\widehat{\m{A}}_k$. The next sections are dedicated to the study of the possible choices for $\m{\Gamma}_k$ and $\widehat{\m{A}}_k$.

\subsection{BCH formulas for $\m{\Gamma}_k$ and $\widehat{\m{A}}_k$} \cref{alg:Zimmermannsalgorithm} chose $\m{\Gamma}_k$ using the Baker-Campbell-Hausdorff (BCH) formula. For $\m{X,Y}$ in the Lie algebra (e.g., $\mathrm{Skew}(n)$) of a Lie group (e.g., $\mathrm{SO}(n)$), the BCH formula gives
\begin{equation}\label{eq:BCHformula}
    \log(\exp(\m{X})\exp(\m{Y})) = \m{X}+\m{Y} + \frac{1}{2}\left[\m{X},\m{Y}\right]+\frac{1}{12}\left[\m{X}-\m{Y},\left[\m{X},\m{Y}\right]\right]+\text{H.O.T(4)},
\end{equation}
where $\left[\m{X},\m{Y}\right] := \m{XY}-\m{YX}$ is called the \emph{Lie bracket} or \emph{commutator}. H.O.T(4) stands for ``Higher Order Terms of 4th order''. 
The term `order' has to be read with care. It refers to terms in $X,Y$ of combined order 4 or higher. For example, $X^3Y$ and $XYXY$ are such $4$th-order terms.

Letting $\m{Z}:=\log(\exp(\m{X})\exp(\m{Y}))$, \cite[Thm.~1]{THOMPSON1989} showed that~\eqref{eq:BCHformula} was convergent for $\|\m{X}\|_2,\|\m{Y}\|_2\leq\mu<1$.
In the upcoming~\cref{cond:delta}, we state locality assumptions that are always sufficient for the convergence of~\eqref{eq:BCHformula}, namely conditions that yield $\|\m{X}\|_2,\|\m{Y}\|_2\leq\mu<1$. In view of the fundamental equation~\eqref{eq:fundamentalequation}, we define
\begin{equation}\label{eq:XandY}
    \m{X}_k = \begin{bmatrix}
                2\beta \m{A}_{k} & -\m{B}_{k}^T\\
                \m{B}_{k}& \m{C}_{k}
            \end{bmatrix} \text{ and } \m{Y}_k=\begin{bmatrix}
                \m{\Theta}_k&0\\
                0&\m{\Gamma}_k
            \end{bmatrix}.
\end{equation}
The BCH series expansion of the $p\times p$ bottom-right block up to fifth order of~\eqref{eq:BCHformula} yields
\begin{align}
    \label{eq:first}
    \m{C}_{k+1} &= \m{C}_{k}+\m{\Gamma}_k-\frac{1}{12}\left(\m{\Gamma}_k\m{B}_k\m{B}_k^T+\m{B}_k\m{B}_k^T\m{\Gamma}_k \right)\\
    \label{eq:highergamma}
    &+\mathcal{O}(\|\m{\Gamma}_k\|_2^i\|\m{C}_k\|_2^j)\\
    \label{eq:thetaterm}
    &+ \frac{1}{6}\m{B}_k\m{\Theta}_k\m{B}_k^T\\
    \label{eq:highertheta}
    &-\frac{1}{12}\left(\m{B}_k\m{\Theta}_k\m{B}_k^T\m{\Gamma}_k-\m{\Gamma}_k\m{B}_k\m{\Theta}_k\m{B}_k^T \right)\\
    \label{eq:hot5}
    &+\text{H.O.T}_{\m{C}}(5),
\end{align}
with $i,j\geq1$ in~\eqref{eq:highergamma}. The objective pursued in \cite{ZimmermannRalf22} to prove the linear convergence of~\cref{alg:Zimmermannsalgorithm} was to show that $\|\m{C}_{k+1}\|_2\leq a\|\m{C}_{k}\|_2 +\mathcal{O}(\|\m{C}_{k}\|_2^2)$ with $0<a<1$. The terms~\eqref{eq:first} and~\eqref{eq:highergamma} are the same as in \cite{ZimmermannRalf22}. The terms~\eqref{eq:first} can thus be cancelled similarly by taking $\m{\Gamma}_k$ as the solution to the \emph{Sylvester equation}~\eqref{eq:sylvesterfirst}. Assuming $\|\m{B}_k\|_2,\|\m{C}_k\|_2<\delta$ with $\delta<1$, $\m{\Gamma}_k$ satisfies
\begin{equation}\label{eq:bound_gamma_by_C}
{\new\|\m{\Gamma}_k\|_2\leq \frac{6}{6-\delta^2}\|\m{C}_k\|_2 \text{ (see \cite[p.967]{ZimmermannRalf22})}.}
\end{equation}
Hence, \eqref{eq:highergamma} turns out to be $\mathcal{O}(\|\m{C}_{k}\|_2^2)$, and is thus neglected.  The commonalities with \cite{ZimmermannRalf22} end here since the other terms are new or different. We need a bound on~\eqref{eq:thetaterm},~\eqref{eq:highertheta} and~\eqref{eq:hot5} in terms of $\|\m{C}_{k}\|_2$. From our experience, it would be a mistake to cancel~\eqref{eq:thetaterm} and/or~\eqref{eq:highertheta} with $\m{\Gamma}_k$ because of the negative influence it would have on~\eqref{eq:hot5}. Hence, the challenge is to choose $\widehat{\m{A}}_k$ where $\|\m{\Theta}_k\|_2$ converges linearly to $0$ when $\|\m{C}_k\|_2$ does. \cref{thm:convergence} shows that the \emph{backward iteration}, i.e., taking $\widehat{\m{A}}_k=\m{A}_k$, verifies this property. We also propose convergent \emph{forward iterations} in~\cref{sec:forwardconvergence} based on~\cref{thm:convergence}.

\subsection{Linear convergence of the backward iteration}\label{sec:convergencebackward} In this section, we generalize the analysis that was done for~\cref{alg:Zimmermannsalgorithm} in \cite{ZimmermannRalf22}, but for~\cref{alg:generalizedStiefel}. That is, assuming convergence,  we prove the local linear convergence of~\cref{alg:generalizedStiefel} implemented with backward iterations. 
In~\cref{sec:forwardconvergence}, we generalize the convergence result to forward iterations. \cref{thm:convergence} asks for a locality assumption as it was done in~\cref{thm:convergenceZimmermann} with $\left\|\left[\begin{smallmatrix}
        \m{A}_k&-\m{B}_k^T\\
        \m{B}_k&\m{C}_k
    \end{smallmatrix}\right]\right\|_2<\delta$. Given the \emph{fundamental equation}~\eqref{eq:fundamentalequation}, it would be natural for us to consider $\left\|\left[\begin{smallmatrix}
        2\beta\m{A}_k&-\m{B}_k^T\\
        \m{B}_k&\m{C}_k
    \end{smallmatrix}\right]\right\|_2<\delta$. It turns out that we need to bound both aforementioned matrices. The first option yields an easier expression for~\cref{thm:convergence} since it allows to write $\|\m{A}_k\|_2<\delta$ instead of $\|\m{A}_k\|_2<\frac{\delta}{2\beta}$. \cref{lem:normofX} shows that these two choices are equivalent up to a multiplicative factor. The initial choice is thus arbitrary in view of the asymptotic behavior.
    \begin{rem}
   Important efforts have been engaged in \cite{Zimmermann17} to obtain a sufficient condition on $\|\widetilde{U}-\m{U}\|_\mathrm{F}$ for the convergence of~\cref{alg:Zimmermannsalgorithm}. This led to $\|\widetilde{U}-\m{U}\|_\mathrm{F}<0.091$ \cite[Thm.~4.1]{Zimmermann17}. This value is not satisfying because it is not representative of the condition observed in practice. We want a notion of radius of convergence that is probabilistic and scales with the size of $\mathrm{St}(n,p)$, e.g., $2\sqrt{p}$, the Frobenius diameter. We propose the following approach. We obtain a theoretical linear rate of convergence for~\cref{alg:generalizedStiefel} in~\cref{thm:convergence} at the cost of feasibility assumptions (\cref{cond:delta}) and then we investigate the convergence radius numerically in~\cref{sec:logisticmodel}. 
    \end{rem}
\cref{thm:convergence} is a generalization of~\cref{thm:convergenceZimmermann}, borrowing voluntarily its format, and reducing to it when $\beta=\frac{1}{2}$ (canonical metric).
\begin{thm}\label{thm:convergence}
    Given {\new $\beta>\frac{1}{4}$, $\m{U},\widetilde{U}\in\mathrm{St}(n,p)$ $(\m{U}\neq\widetilde{U})$ and~\cref{alg:generalizedStiefel} with $\widehat{\m{A}}_k = \m{A}_k$ in step 6. Assume there is $k$ such that $X_k,Y_k$ from \eqref{eq:XandY} satisfy $\|Y_k\|_2\leq\|X_k\|_2$.} Assume further that there is $\delta>0$ satisfying~\cref{cond:delta} and such that for $\m{L}_k:=\begin{bmatrix}
        \m{A}_k&-\m{B}_k^T\\
        \m{B}_k&\m{C}_k
    \end{bmatrix}$ it holds $\|\m{L}_k\|_2<\delta$ throughout the algorithm's iteration loop. Then, it holds that
    \begin{equation}
    \|\m{C}_{k+1}\|_2\leq \underbrace{\left(\frac{\eta\delta^4}{6\xi(6-\delta^2)}+\left(1+\frac{\eta}{\xi}\right)\frac{\kappa\alpha}{1-\frac{\eta\alpha}{\xi}}\right)}_{\mathcal{O}(\delta^4)}\|\m{C}_k\|_2+\mathcal{O}(\|\m{C}_k\|_2^2),
\end{equation}
where the constant factors are defined by
\begin{align*}
    \tau &:= 1-2\beta,\hspace{2.7cm}\eta:= |\tau|\left(1+\delta|\tau|{\new +\frac{2}{3}\delta^2|\tau|^2}-\delta^2|\tau|^2\log\left(1-2\delta|\tau|\right)\right),\\
    \alpha&:=\frac{\delta^4 (1+|\tau|)^4}{1-\delta(1+|\tau|)},\hspace{1.6cm}\kappa:=\max\Big(\frac{6}{6-\delta^2},\frac{\eta\delta^2}{\xi(6-\delta^2)}\Big),\\
    \xi&:=2\beta - \eta\left(1+2\beta\delta+\frac{4\beta^2}{3}\delta^2+\frac{\delta^2}{6}+\frac{\delta^3}{6-\delta^2}\right).
\end{align*}
\end{thm}
\begin{proof}
    The proof consists of bounding~\eqref{eq:thetaterm},~\eqref{eq:highertheta} and~\eqref{eq:hot5} in terms of $\|\m{C}_{k}\|_2$. Although it is quite tedious, we demonstrate that it can be done. Recall that $\|\m{L}_k\|_2<\delta$ implies that $\|\m{A}_k\|_2,\|\m{B}_k\|_2,\|\m{C}_k\|_2<\delta$. {\new This proof involves the series expansions of several matrices with respective higher order terms. To avoid ambiguity, we append subscripts to ``H.O.T'' to distinguish them. There will be various conditions for the well-definedness of the parameters involved. These are only listed collectively in \cref{cond:delta} after the proof, as they would seem unmotivated at this point.} A first milestone is to bound $\|\m{\Theta}_k\|_2$ in terms of $\|\m{C}_{k}\|_2$. The BCH series expansion of $\m{\Theta}_k$, based on~\eqref{eq:thetadefinition}, leads to
\begin{equation}\label{eq:expansion_of_theta}
    \m{\Theta}_k= \tau (\m{A}_k-\m{A}_{k+1})-\frac{\tau^2}{2}[\m{A}_k, \m{A}_{k+1}]-\frac{\tau^3}{12}[{\new \m{A}_k+\m{A}_{k+1}}, [\m{A}_k, \m{A}_{k+1}]]+\text{H.O.T}_{\m{\Theta}}(4).
\end{equation}
By~\cref{lem:bound_hot_theta}, we have $\|\mathrm{H.O.T}_{\m{\Theta}}(4)\|_2\leq |\tau|(|\tau|\delta)^2\log\left(\frac{1}{1-2\delta|\tau|}\right)\|\m{A}_k-\m{A}_{k+1}\|_2=:|\tau|\zeta \|\m{A}_k-\m{A}_{k+1}\|_2$.
If we take the norm on both sides and mind that $[\m{A}_k,\m{A}_{k+1}]=[\m{A}_k,\m{A}_{k+1}-\m{A}_{k}]$, we get a bound on $ \|\m{\Theta}_k\|_2$ in terms of $\|\m{A}_k-\m{A}_{k+1}\|_2$:
\begin{align}
\nonumber
    \|\m{\Theta}_k\|_2 &\leq |\tau|(1+\zeta) \|\m{A}_k-\m{A}_{k+1}\|_2 +\left(\frac{|\tau|^2}{2}+\frac{|\tau|^3}{6}{\new \|\m{A}_k+\m{A}_{k+1}\|_2}\right)\| [\m{A}_k, \m{A}_{k+1}]\|_2\\
    \nonumber
    &\leq |\tau|(1+\zeta) \|\m{A}_k-\m{A}_{k+1}\|_2 +\left(\frac{|\tau|^2}{2}+\frac{|\tau|^3}{6}{\new 2\delta}\right)2\|\m{A}_{k}\|_2\|\m{A}_k-\m{A}_{k+1}\|_2\\
     \label{eq:thetafirstbound}
     &\leq \eta\|\m{A}_k-\m{A}_{k+1}\|_2,
\end{align}
where $\eta := |\tau|(1+\delta|\tau| + {\new \frac{2}{3}\delta^2 |\tau|^2}+\zeta)$. Now, we leverage the top-left $p\times p$ block of \eqref{eq:BCHformula} to bound $\|\m{A}_k-\m{A}_{k+1}\|_2$ in terms of $\|\m{C}_k\|_2$.  This is done in~\cref{lem:partialproofthmconvergence} and yields
\begin{align}
    \label{eq:equationfornormdeltaA}
    2\beta\|\m{A}_k-\m{A}_{k+1}\|_2&\leq \eta\left(1+2\beta\delta+\frac{4\beta^2}{3}\delta^2+\frac{\delta^2}{6}+\frac{\delta^3}{6-\delta^2}\right)\|\m{A}_k-\m{A}_{k+1}\|_2\\
    \nonumber
    &+ \frac{\delta^2}{6-\delta^2}\|\m{C}_k\|_2+\|\text{H.O.T}_{\m{A}}(5)\|_2+ \mathcal{O}(\|\m{A}_k-\m{A}_{k+1}\|_2^2).
\end{align}
By~\cref{cond:delta}, we have $\xi:=2\beta - \eta\left(1+2\beta\delta+\frac{4\beta^2}{3}\delta^2+\frac{\delta^2}{6}+\frac{\delta^3}{6-\delta^2}\right)>0$, which leads to
\begin{align}\label{eq:AkasCk}
    \|\m{A}_k-\m{A}_{k+1}\|_2&\leq \frac{\delta^2}{\xi(6-\delta^2)}\|\m{C}_k\|_2+\frac{1}{\xi}\|\text{H.O.T}_{\m{A}}(5)\|_2+ \mathcal{O}(\|\m{A}_k-\m{A}_{k+1}\|_2^2).
\end{align}
{\new By inserting \eqref{eq:AkasCk} into $\mathcal{O}(\|\m{A}_k-\m{A}_{k+1}\|_2^2)$, it follows that 
\begin{equation*}
\mathcal{O}(\|\m{A}_k-\m{A}_{k+1}\|_2^2)\in\mathcal{O}(\|\m{C}_k\|_2^2, \|\text{H.O.T}_{\m{A}}(5)\|_2^2,\|\m{C}_k\|_2 \|\text{H.O.T}_{\m{A}}(5)\|_2 ).
\end{equation*}
We introduce the short-hand notation $\Psi$ to represent all $\mathcal{O}(\|\m{A}_k-\m{A}_{k+1}\|_2^2)$ terms. In particular, we consider $\alpha\Psi=\Psi$ for all $\alpha>0$.} We will obtain further that $\m{\Psi}\in\mathcal{O}(\|\m{C}_k\|_2^2)$.
Inserting~\eqref{eq:AkasCk} into~\eqref{eq:thetafirstbound} finally yields
\begin{equation}\label{eq:thetawithhot}
    \|\m{\Theta}_k\|_2 \leq \frac{\eta\delta^2}{\xi(6-\delta^2)}\|\m{C}_k\|_2+\frac{\eta}{\xi}\|\text{H.O.T}_{\m{A}}(5)\|_2+ \m{\Psi}.
\end{equation}
Only the higher order terms are still to be eliminated. Notice that by the properties of $\|\cdot\|_2$, the block-expansions $\|\text{H.O.T}_{\m{A}}(5)\|_2$ from~\eqref{eq:thetawithhot} and $\|\text{H.O.T}_{\m{C}}(5)\|_2$ from~\eqref{eq:hot5} are both bounded from above by $\|\text{H.O.T(5)}\|_2$, the higher order terms of the complete series expansion of $\log(\exp(\m{X}_k)\exp(\m{Y}_k))$ from~\eqref{eq:XandY}. By definition of the BCH series expansion of $\m{X}_k$ and $\m{Y}_k$, we have
\begin{equation}\label{eq:boundhot5}
    \|\text{H.O.T(5)}\|_2\leq\sum_{l=5}^\infty\|z_l(\m{X}_k,\m{Y}_k)\|_2,
\end{equation}
where $z_l(\m{X}_k,\m{Y}_k)$ is the sum over all words of length $l$ in the alphabet $\{\m{X}_k,\m{Y}_k\}$ multiplied by their respective Goldberg coefficient of order $l$ \cite{Goldberg1956}. The goal is to obtain a relation between $\|\m{X}_k\|_2,\|\m{Y}_k\|_2$ and $\|\m{C}_k\|_2$. {\new By \eqref{eq:bound_gamma_by_C} and \eqref{eq:thetawithhot}}, it holds
\begin{align*}
    \|\m{Y}_k\|_2 &=\max(\|\m{\Gamma}_k\|_2,\|\m{\Theta}_k\|_2)\\
    &\leq\max\left(\frac{6}{6-\delta^2},\frac{\eta\delta^2}{\xi(6-\delta^2)}\right)\|\m{C}_k\|_2+\frac{\eta}{\xi}\|\text{H.O.T}_{\m{A}}(5)\|_2+\m{\Psi}.
\end{align*}
 It is now convenient to leverage $\|\m{Y}_k\|_2\leq \|\m{X}_k\|_2$ for $k$ large enough because it allows to simplify~\eqref{eq:boundhot5} greatly. It was already used in \cite{ZimmermannRalf22} for the proof of~\cref{thm:convergenceZimmermann}. It is a valid hypothesis since $\m{U}\neq\widetilde{U}$ must lead to $0=\|\m{Y}_\infty\|_2<\|\m{X}_\infty\|_2$ upon convergence.  By~\cref{lem:normofX}, we have $ \|\m{X}_k\|_2\leq (1+|\tau|)\|\m{L}_k\|_2$, {\new where we recall that $L_k=\left[\begin{smallmatrix}A_k&-B_k^T\\B_k&C_k\end{smallmatrix}\right]$}. Defining $\kappa:=\max\Big(\frac{6}{6-\delta^2},\frac{\eta\delta^2}{\xi(6-\delta^2)}\Big)$ and using the property that the sum over all Goldberg coefficients of order~$l$ is less than~$1$ (\cite{THOMPSON1989} and~\cite[Lem.~A.1]{Zimmermann17}), we have

\begin{align*}
     \|\text{H.O.T(5)}\|_2&\leq\sum_{l=5}^\infty\|\m{X}_k\|_2^{l-1}\|\m{Y}_k\|_2\\
     &\leq\sum_{l=5}^\infty(1+|\tau|)^{l-1}\|\m{L}_k\|^{l-1}_2  \left(\kappa\|\m{C}_k\|_2+\frac{\eta}{\xi}\|\text{H.O.T}_{\m{A}}(5)\|_2+\m{\Psi}\right)\\
     &\leq\left(\kappa\|\m{C}_k\|_2+\frac{\eta}{\xi}\|\text{H.O.T(5)}\|_2+\m{\Psi}\right)\sum_{l=5}^\infty(1+|\tau|)^{l-1}\delta^{l-1}\\
     &= \alpha\kappa\|\m{C}_k\|_2+\frac{\eta\alpha}{\xi}\|\text{H.O.T(5)}\|_2+\m{\Psi},
 \end{align*}
 where $\alpha:= \frac{\delta^4 (1+|\tau|)^4}{1-\delta(1+|\tau|)}$ and where the condition $\delta(1+|\tau|)<1$ for the convergence of the series was ensured by~\cref{cond:delta}. Also by~\cref{cond:delta}, we have $1-\frac{\eta\alpha}{\xi}>0$, leading to
 \begin{equation}\label{eq:boundhot}
     \|\text{H.O.T(5)}\|_2 \leq\frac{\kappa\alpha}{1-\frac{\eta\alpha}{\xi}}\|\m{C}_k\|_2 +\m{\Psi}
     \leq\frac{\kappa\alpha}{1-\frac{\eta\alpha}{\xi}}\|\m{C}_k\|_2 +\mathcal{O}(\|\m{C}_k\|_2^2).
 \end{equation}
 Inserting~\eqref{eq:boundhot} in $\m{\Psi}$ shows that $\m{\Psi}\in\mathcal{O}(\|\m{C}_k\|_2^2)$. Finally, we have
\begin{align*}
    \|\m{C}_{k+1}\|_2&\leq \left\|\frac{1}{6}\m{B}_k\m{\Theta}_k\m{B}_k^T\right\|_2+ \|\text{H.O.T}_{\m{C}}(5)\|_2+\mathcal{O}(\|\m{C}_k\|_2^2)\\
    &\leq \frac{\eta\delta^4}{6\xi(6-\delta^2)}\|\m{C}_k\|_2+\left(1+\frac{\eta}{\xi}\right)\|\text{H.O.T(5)}\|_2+\mathcal{O}(\|\m{C}_k\|_2^2)\\
    &\leq \left(\frac{\eta\delta^4}{6\xi(6-\delta^2)}+\left(1+\frac{\eta}{\xi}\right)\frac{\kappa\alpha}{1-\frac{\eta\alpha}{\xi}}\right)\|\m{C}_k\|_2+\mathcal{O}(\|\m{C}_k\|_2^2).
\end{align*}
This concludes the proof.
\end{proof}
As seen above, the proof of~\cref{thm:convergence} requires some technical conditions. These conditions are gathered in~\cref{cond:delta} to be seen at a glance. \cref{fig:condition} illustrates the admissible set of pairs $(\beta,\delta)$ satisfying~\cref{cond:delta}. Since all inequalities are strict, the set is an open set. In particular, notice that the largest feasible $\delta$ occurs at $\beta=\frac{1}{2}$ ($\tau=\eta=\zeta=0,\  \xi = 1,\ \alpha=\frac{\delta^4}{1-\delta},\ \kappa = \frac{6}{6-\delta^2}$), which corresponds to the canonical metric. {\new In~\cref{cond:delta}, we have by definition $\xi\leq 2\beta-\eta\leq 2\beta - |\tau| = 2\beta -|1-2\beta|$. When $\beta\leq\frac{1}{4}$, $2\beta -|1-2\beta|\leq 0$ and the condition $\xi>0$ can never be satisfied.}
\begin{cond}\label{cond:delta}
    Given $\beta,\delta>0$, $\delta$ is admissible for~\cref{thm:convergence} if it satisfies
    \begin{align*}
        &\delta (1+|\tau|)<1,\hspace{1cm}1-\frac{\eta\alpha}{\xi}>0, \hspace{1cm}2\delta|\tau|<1,\\
        &\xi:=2\beta - \eta\left(1+2\beta\delta+\frac{4\beta^2}{3}\delta^2+\frac{\delta^2}{6}+\frac{\delta^3}{6-\delta^2}\right)>0,\\
        &\left(\frac{\eta\delta^4}{6\xi(6-\delta^2)}+\left(1+\frac{\eta}{\xi}\right)\frac{\kappa\alpha}{1-\frac{\eta\alpha}{\xi}}\right)<1,
    \end{align*}
    where 
    $\tau := 1-2\beta,\ \eta := |\tau|(1+\delta|\tau|+ {\new \frac{2}{3}\delta^2 |\tau|^2}-\delta^2|\tau|^2\log\left(1-2|\tau|\delta\right)),\
    \alpha:=\frac{\delta^4 (1+|\tau|)^4}{1-\delta(1+|\tau|)}$ and $
    \kappa:=\max\Big(\frac{6}{6-\delta^2},\frac{\eta\delta^2}{\xi(6-\delta^2)}\Big).$
\end{cond}
\begin{figure}[H]
    \centering
    \includegraphics[width = 6.3cm]{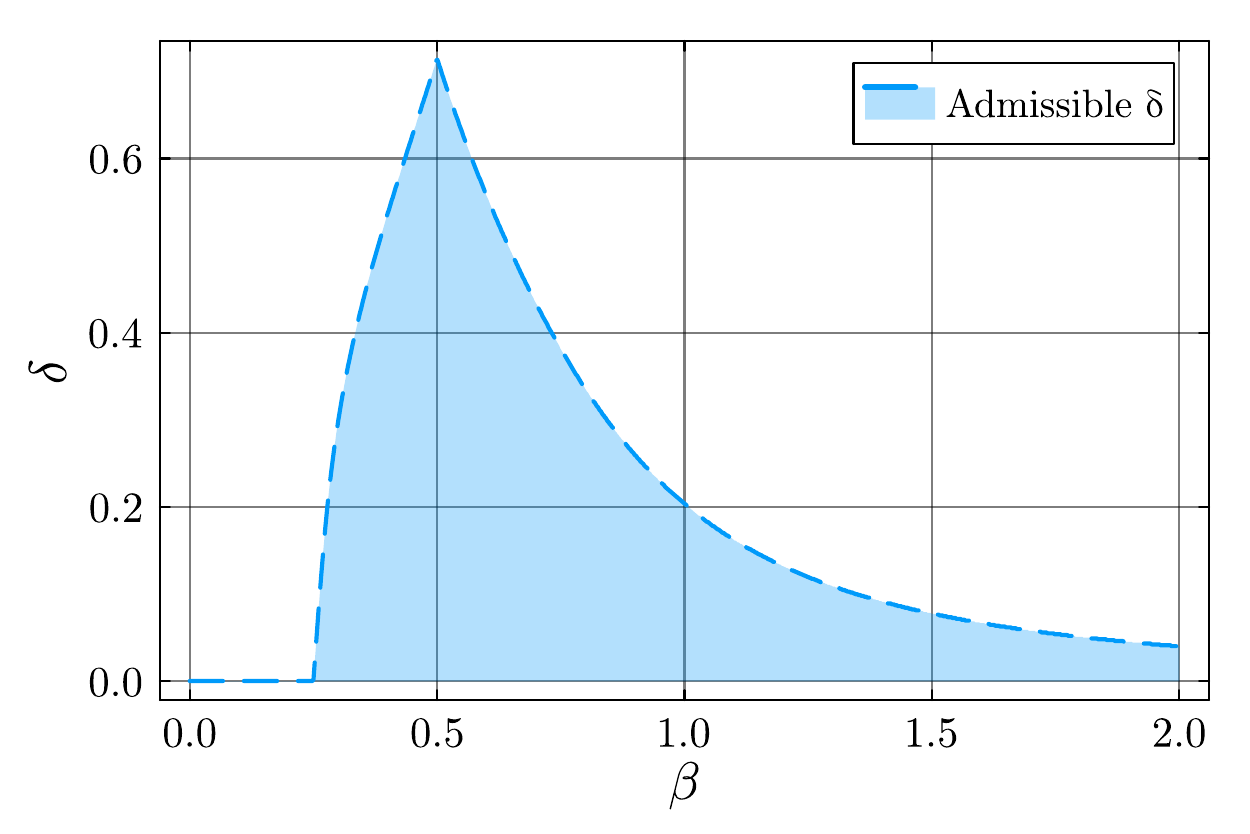}
    \vspace{-0.3cm}
    \caption{Open set of pairs $(\beta,\delta)$ satisfying~\cref{cond:delta}. \cref{cond:delta} gathers sufficient conditions to ensure~\cref{thm:convergence} on the convergence of the backward~\cref{alg:generalizedStiefel}.
    }
    \label{fig:condition}
    \vspace{-0.3cm}
\end{figure}

\subsection{Local linear convergence of three different forward iterations}\label{sec:forwardconvergence} We address now the question of the \emph{linear} convergence of~\cref{alg:generalizedStiefel} when we use an approximation $\widehat{\m{A}}_k\approx\m{A}_k$. As before, we examine the convergence rate only for instances for which the algorithm converges. The difference between the forward and the backward case comes from~\eqref{eq:thetadefinition}, thus depending completely on $\m{\Theta}_k$. Selecting $\widehat{A}_k$ where $\|\widehat{\m{A}}_k\|_2<\delta$ 
and replacing $A_k$ by $\widehat{A}_k$  in \eqref{eq:expansion_of_theta} yields
\begin{equation}\label{eq:forwardthetaexpansion}
    \|\m{\Theta}_k\|_2\leq \eta  \|\widehat{\m{A}}_k-\widehat{\m{A}}_{k+1}\|_2.
\end{equation}
If we can find $c>0$ (typically $c\geq 1$) such that 
\begin{equation}\label{eq:forwardcondition}
    \|\widehat{\m{A}}_k-\widehat{\m{A}}_{k+1}\|_2\leq c\|\m{A}_k-\m{A}_{k+1}\|_2,
\end{equation}
the new forward convergence theorem follows directly by replacing $\eta$ by $\widehat{\eta}:=c\eta$ in~\cref{thm:convergence} and~\cref{cond:delta}. In practice, \eqref{eq:forwardcondition} means that the approximation $\widehat{A}_k$ converges as fast as $A_k$ does. 
Obtaining $c$ explicitly in~\eqref{eq:forwardcondition} happens to be a burden for many approximations $\widehat{\m{A}}_k$. However, ensuring the  existence of $c\in[0,+\infty)$ is easier. In the next subsections, we propose three linearly convergent ways of choosing $\widehat{\m{A}}_k$.
\subsubsection{The fixed forward iteration}\label{subsec:pureforward} The computationally cheapest choice for the approximation is $\widehat{\m{A}}_k = \m{A}_{k-1}$, which we term \emph{fixed forward iteration}. In this case,~\eqref{eq:forwardcondition} becomes
\begin{equation}\label{eq:reversedconvergence}
    \|\m{A}_{k-1}-\m{A}_{k}\|_2\leq c\|\m{A}_k-\m{A}_{k+1}\|_2,
\end{equation}
Equation~\eqref{eq:reversedconvergence} is reversed when compared to the definition of linear convergence. {\new Indeed, linear convergence is characterized by the relation $\|\m{A}_k-\m{A}_{k+1}\|_2\leq r \|\m{A}_{k-1}-\m{A}_{k}\|_2 $ for some $r\in(0,1)$}. We will show that~\eqref{eq:reversedconvergence} holds for $k\in\mathbb{N}$ with the constant
 \begin{equation}\label{eq:cbar}
     \overline{c} := \max\left(\sup_{k\in\mathbb{N}}\frac{\|\m{A}_{k-1}-\m{A}_{k}\|_2}{\|\m{A}_k-\m{A}_{k+1}\|_2}, \frac{\|\widehat{\m{A}}_{0}-\m{A}_{0}\|_2}{\|\m{A}_0-\m{A}_{1}\|_2}\right)\in [0,+\infty).
 \end{equation} 
 The second argument of the $\max$ is only introduced for~\cref{subsec:acceleratedforward}. Assume there is no static iteration before convergence, i.e., no $k$ such that $\m{A}_{k}=\m{A}_{k+1}$. The property $\overline{c}\geq0$ is a direct consequence of its definition. The fact that $\overline{c}<+\infty$ holds if and only if there is no super-linear convergence of the sequence $\{\|\m{A}_k-\m{A}_{k+1}\|_2\}_{k\in\mathbb{N}}$ to $0$.
 In case this sequence converges super-linearly to $0$, the sequence $\{\|\m{\Theta}_k\|_2\}_{k\in\mathbb{N}}$ does too, by~\eqref{eq:forwardthetaexpansion}. Then, $\m{\Theta}_k$ can simply be ignored from our analysis and we are done. We can thus focus on the case of no super-linear convergence of $\{\|\m{A}_k-\m{A}_{k+1}\|_2\}_{k\in\mathbb{N}}$, which ensures  the existence of $\overline{c}\in[0,+\infty)$. The caveat of the fixed forward iteration is that we have no control on $\overline{c}$ | that is, on the convergence rate. It is fixed by the inputs of \cref{alg:generalizedStiefel} and we can only observe the consequences. This motivates looking for alternative forward iterations.
 In the next subsection, we introduce such an alternative, termed \emph{pseudo-backward}.
\begin{rem}
    For all types of forward iterations,  the first estimate $\widehat{\m{A}}_{0}$ must be chosen by a given rule. We propose to compute it using a BCH series expansion in~\cref{sec:forwardestimate}.
\end{rem}
\subsubsection{The pseudo-backward iteration \texorpdfstring{($c=1+\nu+\nu \overline{c}_\nu$ in \eqref{eq:forwardcondition})}{(4.23)}}\label{subsec:pseudobackward} By a pseudo-backward iteration, we mean an iteration scheme where $\widehat{\m{A}}_{k}$ satifies $\|\widehat{\m{A}}_{k}-\m{A}_{k}\|_2\leq \nu\|\m{A}_k-\m{A}_{k-1}\|_2$ with $\nu\in[0,1]$.  Said otherwise, $\widehat{\m{A}}_{k}$ ensures a sufficient improvement in each iteration compared to the fixed forward approximation $\m{A}_{k-1}$. This intentionally general definition matches very well all cases where we approximately obtain the backward iterate using an iterative method, for instance~\cref{alg:subproblem}. This intermediate type of iteration is designed to converge faster than the fixed forward case without requiring the exact computation of the backward iterate $\m{A}_k$.  Notice that $\nu=1$ corresponds to a fixed forward iteration while $\nu=0$ corresponds to a backward iteration. This new iteration leads to
    \begin{align}
        \label{eq:pseudobackwardconvergence}
        \|\widehat{\m{A}}_k-\widehat{\m{A}}_{k+1}\|_2 &= \|\widehat{\m{A}}_k-\widehat{\m{A}}_{k+1}+\m{A}_k-\m{A}_{k+1}-\m{A}_k+\m{A}_{k+1}\|_2\\
        \nonumber
        &\leq\|\widehat{\m{A}}_{k}-\m{A}_{k}\|_2 + \|\widehat{\m{A}}_{k+1}-\m{A}_{k+1}\|_2+\|\m{A}_{k+1}-\m{A}_{k}\|_2\\
        \nonumber
        &\leq \nu\|\m{A}_{k}-\m{A}_{k-1}\|_2 + \nu\|\m{A}_{k+1}-\m{A}_{k}\|_2+\|\m{A}_{k+1}-\m{A}_{k}\|_2\\
        \nonumber
        &\leq (1+ \nu + \nu \overline{c}_\nu)\|\m{A}_{k+1}-\m{A}_{k}\|_2,
    \end{align}
    where $\overline{c}_\nu$ is defined as in~\eqref{eq:cbar}, but is affected by $\nu$. The goal here is to enhance the convergence rate by moderating the constant $\overline{c}_\nu$ by the adaptive factor $\nu$ such that $1+\nu+\nu \overline{c}_\nu\ll \overline{c}$. We design a pseudo-backward algorithm in~\cref{sec:pseudobackward}.
\subsubsection{The accelerated forward iteration \texorpdfstring{($c=\frac{(1+|h|)\overline{c}_h}{1-\overline{c}_h|h|}$ in \eqref{eq:forwardcondition})}{(4.23)}}\label{subsec:acceleratedforward} The goal of this third and last type of iteration is to build an improved approximation $\widehat{\m{A}}_k$ using the information encoded in the previous-stage gap  $\m{A}_{k-1}-\widehat{\m{A}}_{k-1}$. We define the accelerated forward iteration by 
\begin{equation}\label{eq:accelerated_def}
\widehat{\m{A}}_k = \m{A}_{k-1}+h\m{Q}_{k-1} (\m{A}_{k-1}-\widehat{\m{A}}_{k-1})\m{Q}_{k-1}^T,
\end{equation}
 where $h\in\mathbb{R}$ and $\m{Q}_{k-1}\in\mathrm{O}(p)$. {\new The step size $h$ controls the importance given to the momentum while the matrix $Q_{k-1}$ performs a change of basis that aims at accounting that $A_{k-1}$ and $\m{A}_{k-1}-\widehat{\m{A}}_{k-1}$ are expressed in different bases. A concrete example is given at the end of this section.}
This new approximant $\widehat{\m{A}}_k$ should be consistent, i.e., the sequence $\{\|\widehat{\m{A}}_k-\m{A}_k\|_2\}_{k\in\mathbb{N}}$ should converge linearly to $0$ when $\{\|\m{A}_{k}-\m{A}_{k-1}\|_2\}_{k\in\mathbb{N}}$ does. This was true by definition in the two previous cases. For the accelerated forward iteration, it follows from
\begin{align*}
    \|\widehat{\m{A}}_k-\m{A}_k\|_2&\leq\|\m{A}_k-\m{A}_{k-1}\|_2 +|h| \|\m{Q}_{k-1}(\widehat{\m{A}}_{k-1}-\m{A}_{k-1})\m{Q}_{k-1}^T\|_2\\
    &=\|\m{A}_k-\m{A}_{k-1}\|_2 +|h| \|\widehat{\m{A}}_{k-1}-\m{A}_{k-1}\|_2\\
    &\leq\sum_{l=0}^{k-1} |h|^{l}\|\m{A}_{k-l}-\m{A}_{k-1-l}\|_2+|h|^k\|\widehat{A}_0-A_0\|_2\hspace{1.9cm}\text{(by recursion)}
    \end{align*}
\begin{align*}
    \hspace{1.8cm} &\leq\|\m{A}_k-\m{A}_{k-1}\|_2\sum_{l=0}^{k} \overline{c}_h^{l}|h|^{l}\hspace{3.7cm}\text{(by \eqref{eq:reversedconvergence} and \eqref{eq:cbar})}\\
    &\leq \frac{1}{1-\overline{c}_h|h|}\|\m{A}_k-\m{A}_{k-1}\|_2\hspace{5cm}\text{(if $\overline{c}_h|h|<1$)}
\end{align*}
Here, $\overline{c}_h$ is defined as in~\eqref{eq:cbar} but depends on $h$. Therefore, the accelerated forward iteration is consistent for $h$ small enough. In addition to consistency, $\{\widehat{A}_k\}_{k\in\mathbb{N}}$ should be convergent by satisfying~\eqref{eq:forwardcondition}. This holds since
\begin{align*}
    \nonumber
    \|\widehat{\m{A}}_{k+1}-\widehat{\m{A}}_{k}\|_2 &= \|\m{A}_{k}-\widehat{\m{A}}_{k}-h\m{Q}_{k}(\widehat{\m{A}}_{k}-\m{A}_{k})\m{Q}_{k}^T\|_2\\
    \nonumber
    &\leq (1+|h|)  \|\widehat{\m{A}}_k-\m{A}_k\|_2\\
    &\leq \frac{1+|h|}{1-\overline{c}_h|h|}\|\m{A}_k-\m{A}_{k-1}\|_2\\
    &\leq \frac{(1+|h|)\overline{c}_h}{1-\overline{c}_h|h|}\|\m{A}_{k+1}-\m{A}_{k}\|_2\hspace{4cm}\text{(if $\overline{c}_h|h|<1$)}
\end{align*}
 Notice that for $h=0$, we retrieve $c=\overline{c}$, the fixed forward iteration. Since $\overline{c}_h$ continuously depends on $h$ and since $\overline{c}_h h=0$ for $h=0$, there is an open neighbourhood of $h=0$ for which $\overline{c}_h|h|<1$. The constant $c= \frac{(1+|h|)\overline{c}_h}{1-\overline{c}_h|h|}$ is worse than $c=\overline{c}$, obtained for the fixed forward iteration. This is because a bad step size $h$, notably a wrong sign of $h$, can make~\cref{alg:generalizedStiefel} converge slower. 
{\newrev
\paragraph{Choosing $\m{Q}_k$ and $h$} We propose an accelerated forward iteration achieving fast numerical convergence, as will be confirmed in \cref{sec:fastforward}. To this end, let  $\tau\coloneq 1-2\beta$ and set $h\coloneq-\tau$, $Q_k\coloneq\exp(-\tau A_k)$ so that \eqref{eq:accelerated_def} becomes
\begin{equation} \label{eq:def_AFI}
	\widehat{A}_{k+1} = A_k - \tau \exp(-\tau A_k)(A_k - \widehat{A}_k)\exp(\tau A_k).
\end{equation}
This choice is based on a geometric motivation, illustrated in \cref{fig:AFI_update}. By extrapolating the shift of the approximation from $\widehat{A}_k$ to $A_k$, we update the current iterate $A_k$ using the parallel transport at zero time of a vector field $\Delta_k(t)$ along the geodesic $\gamma:[0,1]\mapsto\mathrm{SO}(p):t\rightarrow \exp(- \tau A_k t)$, see, e.g., \cite[p.~9]{EdelmanArias98}. Recall that any such vector field along $\gamma$ can be parameterized by $\Delta_k(t) = \exp(- \tau A_k t) W_k(t)$ where $ W_k(t)\in\mathrm{Skew}(p)$. Here, we choose the momentum $\Delta_k(t)$ at unit time by 
\begin{equation}
\label{eq:momentum_def}
	\Delta_k(1) \coloneq  \exp(-\tau A_k) W_k(1)\coloneq   \exp(-\tau A_k)\log\left(\exp(- \tau A_k)\exp(\tau \widehat{A}_k)\right).
\end{equation}
\begin{figure}
	\centering
	\label{fig:AFI_update}
	\includegraphics[width = 12cm]{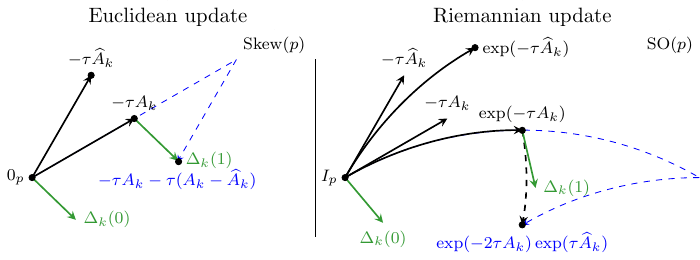}
	\vspace{-0.3cm}
	\caption{An artist view of the geometric construction of the momentum $\Delta_k$. On the left, the figure shows a construction that is equivalent to \eqref{eq:momentum_def} in the context of a flat space, e.g., $\mathrm{Skew}(p)$ equipped with the standard inner product $\langle\cdot,\cdot\rangle_\mathrm{F}$. On the right, the figure illustrates the analogous construction of $\Delta_k$  in the context of the Riemannian manifold $\mathrm{SO}(p)$ viewed as a subset of $(\mathrm{St}(p, p)$,$\langle\cdot,\cdot\rangle_{\beta=1})$.}
	\vspace{-0.5cm}
\end{figure}However, since $\Delta_k(1)\in T_{\exp(-\tau A_k)}\mathrm{SO}(p)$ and $A_k\in T_{I_p}\mathrm{SO}(p)$, these two quantities cannot be directly summed: we must first compute the parallel transport of $\Delta_k(1)$ to zero time, $\Delta_k(0)\in T_{I_p}\mathrm{SO}(p)$, before updating $\widehat{A}_{k+1} = A_k + \alpha_k W_k(0)$ for some step size $\alpha_k\in\mathbb{R}$. It follows from \cite[Eq.~2.18]{EdelmanArias98} that the parallel transport is given by
\begin{equation}\label{eq:transport_def}
	\Delta_k(1)= I_p  \exp\left(-\frac{\tau A_k}{2}\right) W_k(0) \exp\left(-\frac{\tau A_k}{2}\right).
\end{equation}
By combining \eqref{eq:momentum_def} and \eqref{eq:transport_def}, we have $W_k(0) =  \exp\left(-\frac{\tau A_k}{2}\right)  W_k(1)\exp\left(\frac{\tau A_k}{2}\right) $ and, therefore, the update is
\begin{align}
	\nonumber
	\widehat{A}_{k+1} =& A_k + \alpha_k \exp\left(-\frac{\tau A_k}{2}\right) \log\left[\exp(- \tau A_k)\exp(\tau \widehat{A}_k)\right]\exp\left(\frac{\tau A_k}{2}\right) \\
	\label{eq:modified_update}
	=&  A_k + \alpha_k \exp\left(-\tau A_k\left(\frac{1}{2}+\omega\right)\right)\cdot\\
	\nonumber
	& \log\left[ \exp\left(-\tau A_k(1-\omega)\right)\exp(\tau \widehat{A}_k) \exp\left(-\omega\tau A_k\right)\right]\exp\left(\tau A_k\left(\frac{1}{2}+\omega\right)\right),
\end{align}
where $\omega\in\mathbb{R}$ is a free parameter. The purpose of introducing $\omega$ is that the numerical scheme of~\eqref{eq:def_AFI} arises by choosing $\alpha_k=1$ and by linearizing the costly matrix logarithm of \eqref{eq:modified_update} using the BCH formula by
\begin{align}
	\label{eq:linearized_log}
	 \log\left[ \exp\left(-\tau A_k(1-\omega)\right) \exp(\tau \widehat{A}_k) \exp\left(-\omega\tau A_k\right)\right]& =\\
	 \nonumber
	 & \hspace{-4cm} -\tau (A_k- \widehat{A}_k)+\frac{\tau^2(1-2\omega)}{2}[\widehat{A}_k, A_k] +\ \text{H.O.T}(3).
\end{align}
Equation \eqref{eq:linearized_log} shows that choosing $\omega=\frac{1}{2}$ yields a linearization of higher-order accuracy. In addition to be cheaper than~\eqref{eq:modified_update}, experiments suggest that \eqref{eq:def_AFI} achieves the same performance as~\eqref{eq:modified_update} in terms of number of iterations of \cref{alg:generalizedStiefel}. A likely explanation is that the higher accuracy of \eqref{eq:modified_update} is lost throughout the iteration loop of \cref{alg:generalizedStiefel}. Current attempts to demonstrate theoretically the effectiveness of this choice of accelerated forward iteration remained unfruitful due to the several layers of nonlinearities that are involved. 
}

 \subsection{An efficient start point for the forward iteration}\label{sec:forwardestimate}
   A good initial guess for $\widehat{\m{A}}_0$ is important for the fast convergence of~\cref{alg:generalizedStiefel}. Starting from~\eqref{eq:tacklingdifficulty} and defining $\log(\m{V}_{0}):=\begin{bmatrix}
       \m{E}&-\m{F}^T\\
       \m{F}&\m{G}
   \end{bmatrix}\in\mathrm{Skew}(2p)$, the BCH series expansion of the $p\times p$ top-left block of~\eqref{eq:tacklingdifficulty} implies that
   \begin{equation}
        \label{eq:termtocancelforestimate}
       2\beta(\m{A}_{0} - \widehat{\m{A}}_{0}) = \m{E} - \widehat{\m{A}}_{0} {\new + \frac{\tau}{12}(\m{F}^T\m{F} \widehat{\m{A}}_{0}+\widehat{\m{A}}_{0}\m{F}^T\m{F})}-\frac{\tau}{2}[\m{E}, \widehat{\m{A}}_{0}]+\text{H.O.T.}
   \end{equation}
   To obtain a good approximation $ \widehat{\m{A}}_{0}\approx \m{A}_{0}$, we choose $ \widehat{\m{A}}_{0}$  in such a way that the three first right terms of~\eqref{eq:termtocancelforestimate} cancel out. Hence, $\widehat{\m{A}}_{0}$ must be a solution of the Sylvester equation
   \begin{equation}\label{eq:initialestimate}
       \m{S}\widehat{\m{A}}_{0}+\widehat{\m{A}}_{0}\m{S} = \m{E} \text{ where } \m{S}:=\frac{\m{I}_p}{2} - \frac{\tau}{12}\m{F}^T\m{F}.
   \end{equation}
   If $\beta\geq\frac{1}{2} \ (\tau \leq 0)$, the smallest eigenvalue of $S$ is bounded from below by $\frac{1}{2}$. It follows from~\cite[Thm.~VII.2.12]{bhatia97}, that $\|\widehat{A}_0\|_2\leq \|E\|_2\leq \delta$. If $0<\beta<\frac{1}{2} \ (\tau > 0)$, the smallest eigenvalue of $S$ is bounded from below by $\frac{1}{2} - \frac{\tau}{12}\delta^2$, thus $\|\widehat{A}_0\|_2\leq \frac{6}{6-\tau\delta^2}\|E\|_2\leq \frac{6\delta}{6-\tau\delta^2}$.
   In this paper, all the experiments are performed using this first initial guess $\widehat{\m{A}}_{0}$.
\subsection{A quasi-geodesic sub-problem for the backward iteration}\label{sec:subproblemalgorithm} We have not yet addressed the question of how to perform a \emph{backward iteration}. Equation \eqref{eq:sequenceforbeta} has to be solved. Following \cite{Bendokat_2021}, it can be termed ``finding a \emph{quasi-geodesic}'', stated in~\cref{prob:subproblem}.
\begin{prob}{\textsc{Quasi-geodesic sub-problem}}\label{prob:subproblem}
    Given $\m{V}\in \mathrm{SO}(2p)$ and $\beta>0$, find $\m{D}$, $\m{C}\in\mathrm{Skew}(p),\ \m{B}\in\mathbb{R}^{p\times p}$ such that 
\begin{equation}\label{eq:subproblemequation}
    \gamma(t) \coloneq \exp\left(t\begin{bmatrix}
        2\beta\m{D}&-\m{B}^T\\
        \m{B}&\m{C}
    \end{bmatrix}\right)\begin{bmatrix}
        \exp(t(1-2\beta)\m{D})&\m{0}\\
        \m{0}&\m{I}_p
    \end{bmatrix} \text{ ends in } \gamma(1)= \m{V}.
\end{equation}
\end{prob}
\cref{alg:subproblem} is a method to solve~\cref{prob:subproblem}. It is a simplified version of~\cref{alg:generalizedStiefel} where $\m{C}_k$ is not constrained to converge to $\m{0}$ anymore. Starting from an initial guess $\widehat{D}_0$ ($A_{k-1}$ in practice), we successively obtain $D_k$ and update $\widehat{D}_{k+1}$ as the accelerated forward approximation of $D_{k+1}$ for $k=0,1,...$. Then, we set $\widehat{A}_k = D_\infty$ and perform the next backward iteration of~\cref{alg:generalizedStiefel}. Solving~\cref{prob:subproblem} to $\varepsilon$-precision using~\cref{alg:subproblem} is however not competitive with the previously proposed methods \cite{BrynerDarshan17,sutti2023shooting,ZimmermannRalf22}. This is why we investigate a \emph{pseudo-backward} iteration in~\cref{sec:pseudobackward} | we only perform a few iterations of~\cref{alg:subproblem} to improve the approximation $\widehat{A}_k$. An alternative method to solve~\cref{prob:subproblem} based on shooting principle from \cite{BrynerDarshan17} is proposed in~\cref{app:shootingsubproblem}. In practice, we observed the better performance of~\cref{alg:subproblem}.
\begin{algorithm}
    \caption{The sub-problem's iterative algorithm}\label{alg:subproblem}
    \begin{algorithmic}[1]
        \STATE \textbf{INPUT:} Given $\m{V}\in\mathrm{SO}(2p)$, $\widehat{D}_0\in\mathrm{Skew}(p)$, $\beta>0$ and $\varepsilon>0$, compute:
        \STATE Define $\tau = 1-2\beta$.
        \FOR{$k=0,1,...$}
            \STATE Compute $\begin{bmatrix}
                2\beta\m{D}_k&-\m{B}_k^T\\
                \m{B}_k&\m{C}_k
            \end{bmatrix}=\log\left(\m{V}\begin{bmatrix}
                \exp(-\tau\widehat{\m{D}}_k)&\m{0}\\
                \m{0}&\m{I}_p
            \end{bmatrix}\right)$.
            \IF{$\|\m{D}_k-\widehat{\m{D}}_k\|_\mathrm{F}<\varepsilon$}
                \STATE Break.
            \ENDIF
            \STATE Set $\widehat{\m{D}}_{k+1} = \m{D}_k - \tau\exp(-\tau\m{D}_k)(\m{D}_k-\widehat{\m{D}}_k)\exp(\tau \m{D}_k) $.
        \ENDFOR
        \RETURN $\m{D}_k$
    \end{algorithmic}
\end{algorithm}
\vspace{-0.5cm}
 
\section{The performance of forward iterations}\label{sec:forwardperformance} This section investigates the numerical convergence of the different forward iterations introduced in~\cref{sec:forwardconvergence}. These variants of~\cref{alg:generalizedStiefel} are benchmarked in~\cref{sec:benchmark}.
 
\subsection{The accelerated forward iteration}\label{sec:fastforward} \Cref{subsec:acceleratedforward} defined the accelerated forward iteration. \cref{fig:fastforward} quantifies how much this strategy speeds-up~\cref{alg:generalizedStiefel} compared to using the fixed forward iteration. First, notice in~\cref{fig:fastforward} that the convergence rate of the fixed forward iteration is very sensitive to~$\beta$. In comparison, the accelerated forward method provides a convergence rate that is almost independent of $\beta$ when it converges. The improvement effect as a function of $\beta$ and $\|\widetilde{U}-U\|_F$ is summarized on the bottom right plot of~\cref{fig:fastforward}. The fast increase of the improvement factor for $\beta=1$ when the Frobenius distance gets larger is due to the increased convergence radius of the accelerated forward iteration, while the fixed forward iteration gets close to divergence. 
 \begin{figure}
     \centering
     \includegraphics[width=6.3cm]{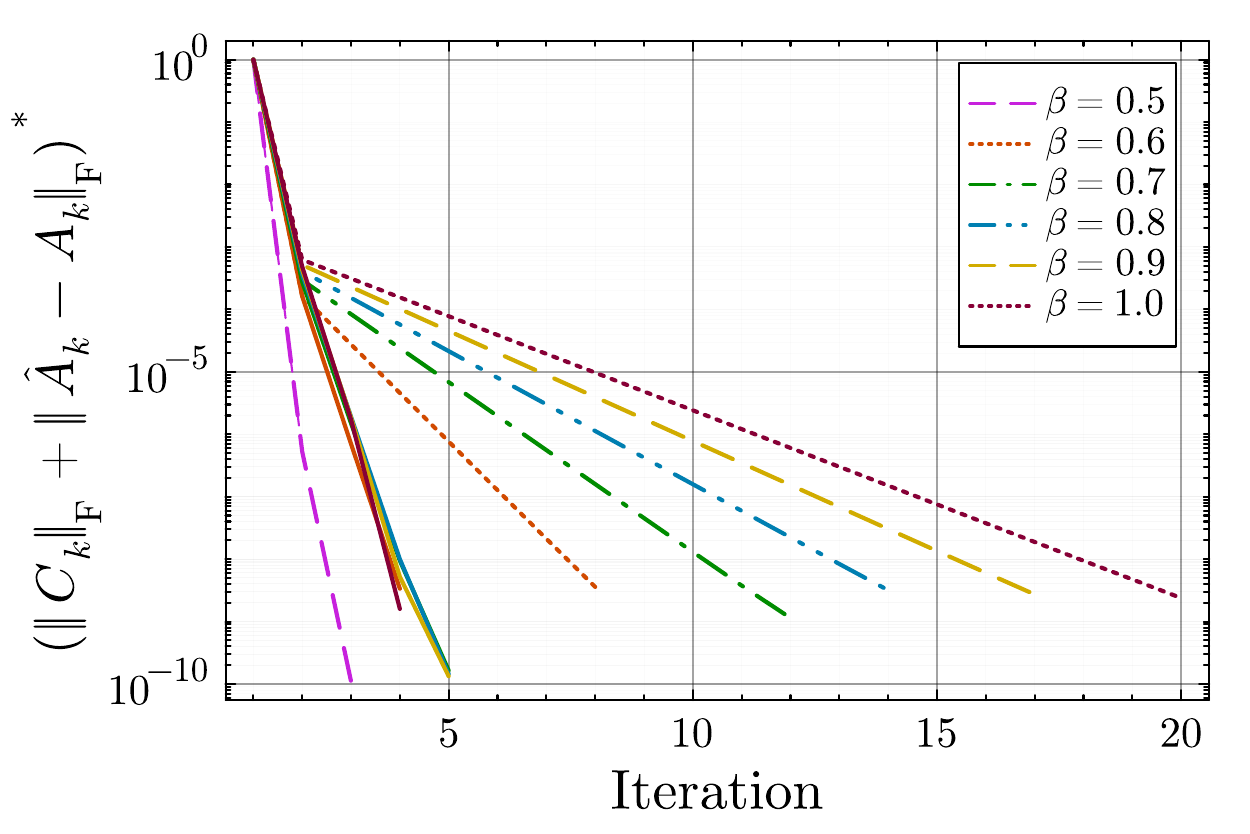}
     \includegraphics[width=6.3cm]{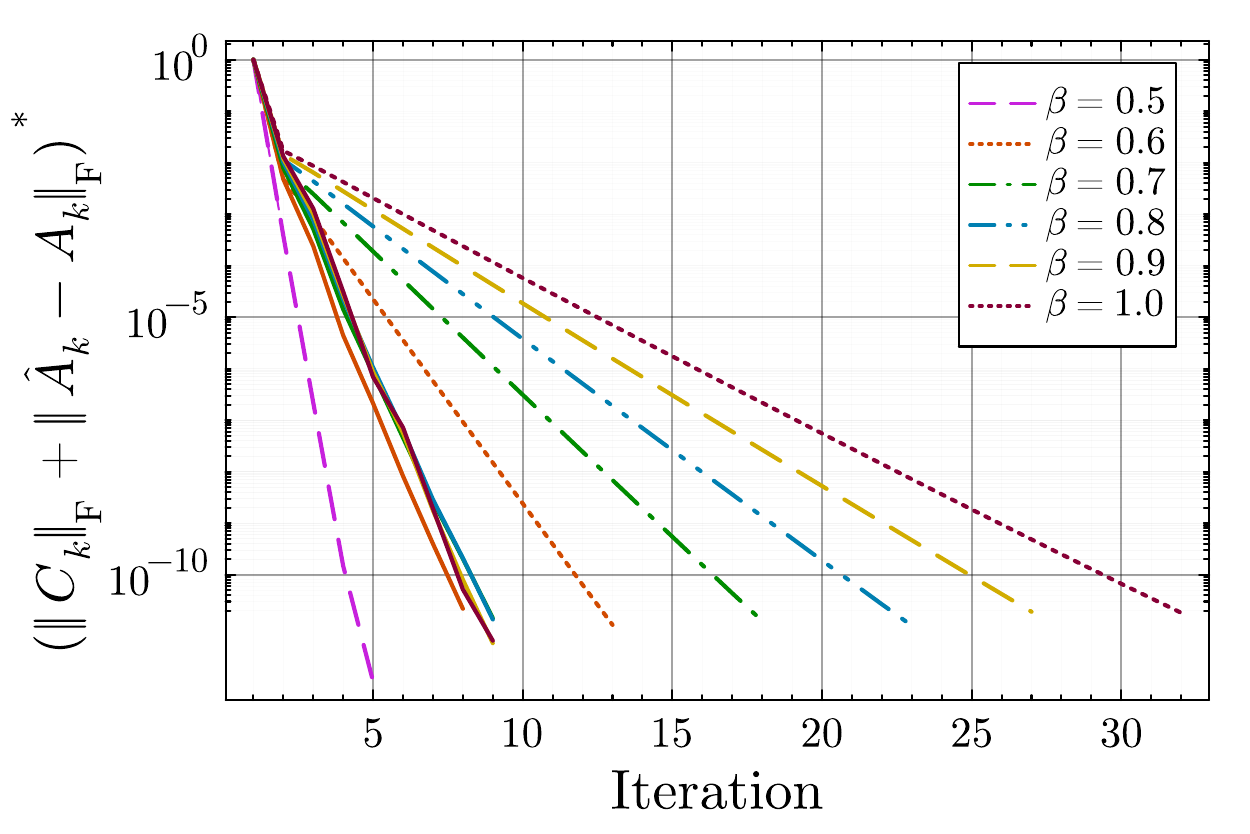}
     \includegraphics[width=6.3cm]{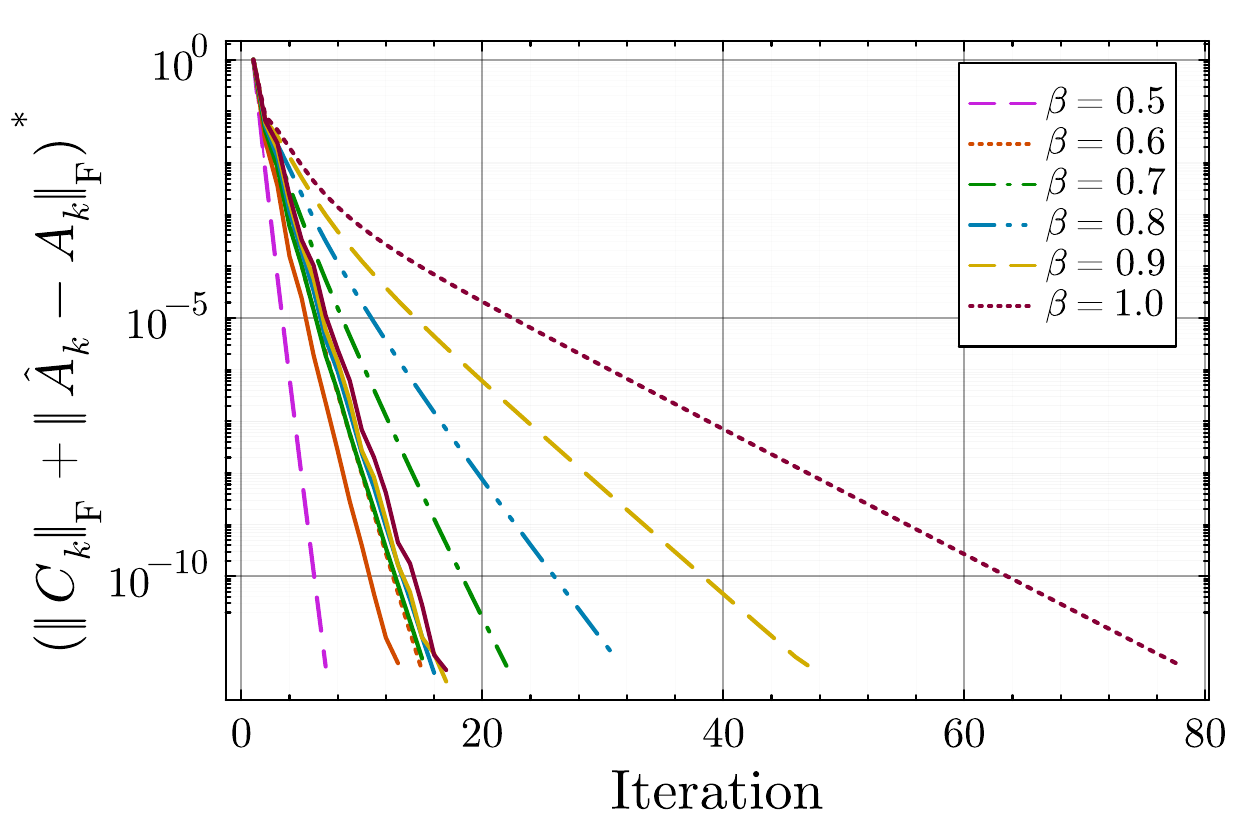}
     \includegraphics[width=6.3cm]{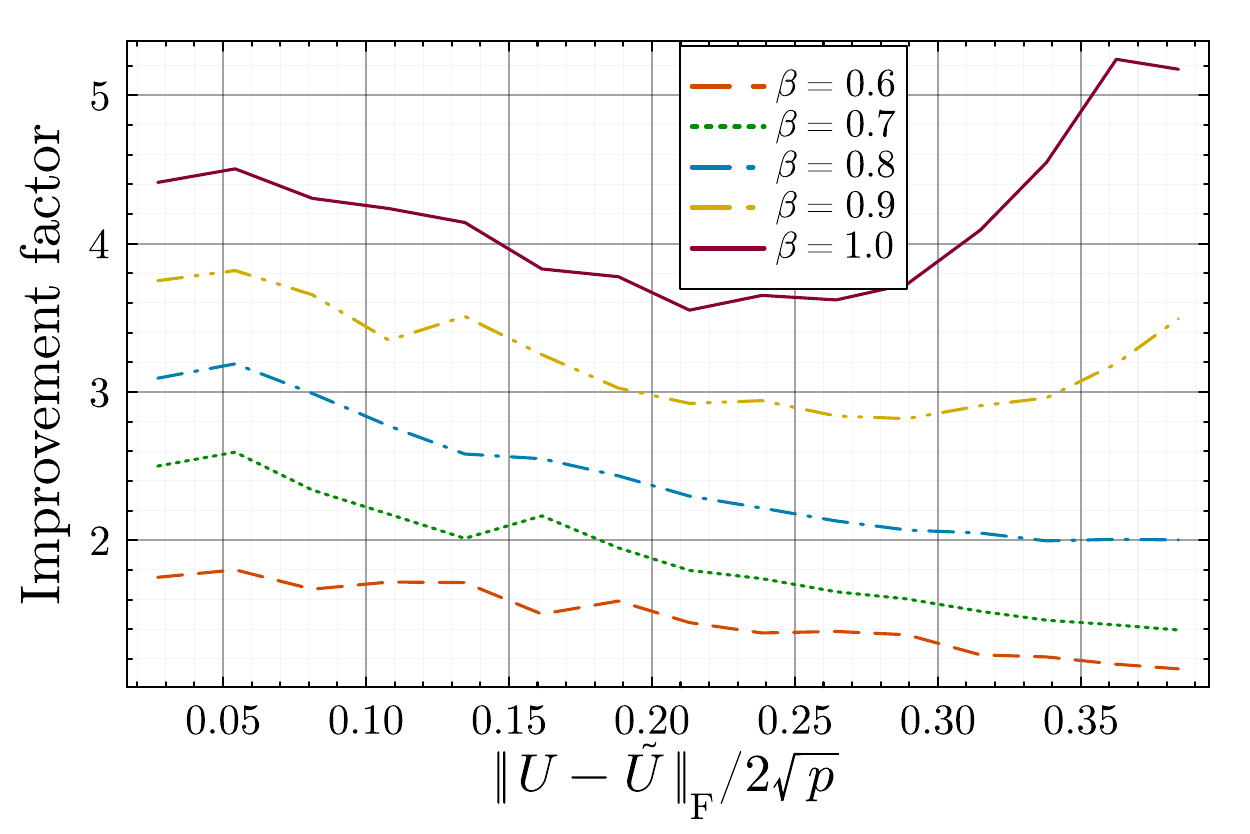}
     \vspace{-0.3cm}
     \caption{Convergence of~\cref{alg:generalizedStiefel} with an accelerated forward iteration ($\m{Q}_k = \exp((2\beta-1)\m{A}_k)$) from~\cref{sec:fastforward} (solid lines) and the fixed forward iteration (stylized lines) on $\mathrm{St}(n=120,p=50)$. The matrices $\m{U},\widetilde{U}$ are randomly generated at Frobenius distance $\|\m{U}-\widetilde{U}\|_\mathrm{F} \in \{0.03,0.19,0.37\}\cdot2\sqrt{p}$ for respectively the top left, top right and bottom left plots. The stars `` $*$'' on the y-axes specify that the residuals are normalized by the residual of the first iteration. The bottom right figure shows how the improvement factor (i.e., the ratio between the number of iterations of the fixed forward and the accelerated forward method) varies as the Frobenius distance increases in $\mathrm{St}(n=60, p=30)$.}
     \label{fig:fastforward}
     \vspace{-0.5cm}
 \end{figure}
 
\subsection{The pseudo-backward iteration compared to the fixed forward iteration}\label{sec:pseudobackward} Pseudo-backward iterations only perform a few sub-iterations of~\cref{alg:subproblem} with $\widehat{\m{D}}_0 := \m{A}_{k-1}$  to improve the quality of the estimation $\widehat{A}_k$ such that $\|\widehat{A}_k-A_k\|_F\leq \nu \|A_{k-1}-A_k\|_F$ with $\nu<1$. The left plot of~\cref{fig:subproblemconvergence} exemplifies that performing two sub-iterations can already provide $\nu \approx 0.1$. Combine this result with a second observation:~\cref{alg:generalizedStiefel} will produce a new iterate $\m{A}_{k+1}$ anyway so that it is not worth the computational effort to obtain $\widehat{\m{A}}_k$ with $\|\widehat{\m{A}}_k-\m{A}_k\|_\mathrm{F}\ll\|\m{A}_{k+1}-\m{A}_k\|_\mathrm{F}$. Hence, the first iterations of~\cref{alg:subproblem} speed-up~\cref{alg:generalizedStiefel} but the last ones are a waste of resource. \cref{fig:subproblemconvergence} investigates the optimal number of sub-iterations to perform. The further $\beta$ is from $\frac{1}{2}$, the more the sub-iterations are improving the performance. {\new For $\beta\in\left[0.7,1\right]$}, the optimal number of sub-iterations settles at $2$.
 
\cref{fig:subiterinfluence} illustrates that the more sub-iterations of~\cref{alg:subproblem} are performed, the fewer iterations of~\cref{alg:generalizedStiefel} are needed. For these experiments, 2 sub-iterations are enough to reach the performance of the backward iteration.

 \begin{figure}
     \centering
     \includegraphics[width = 6cm]{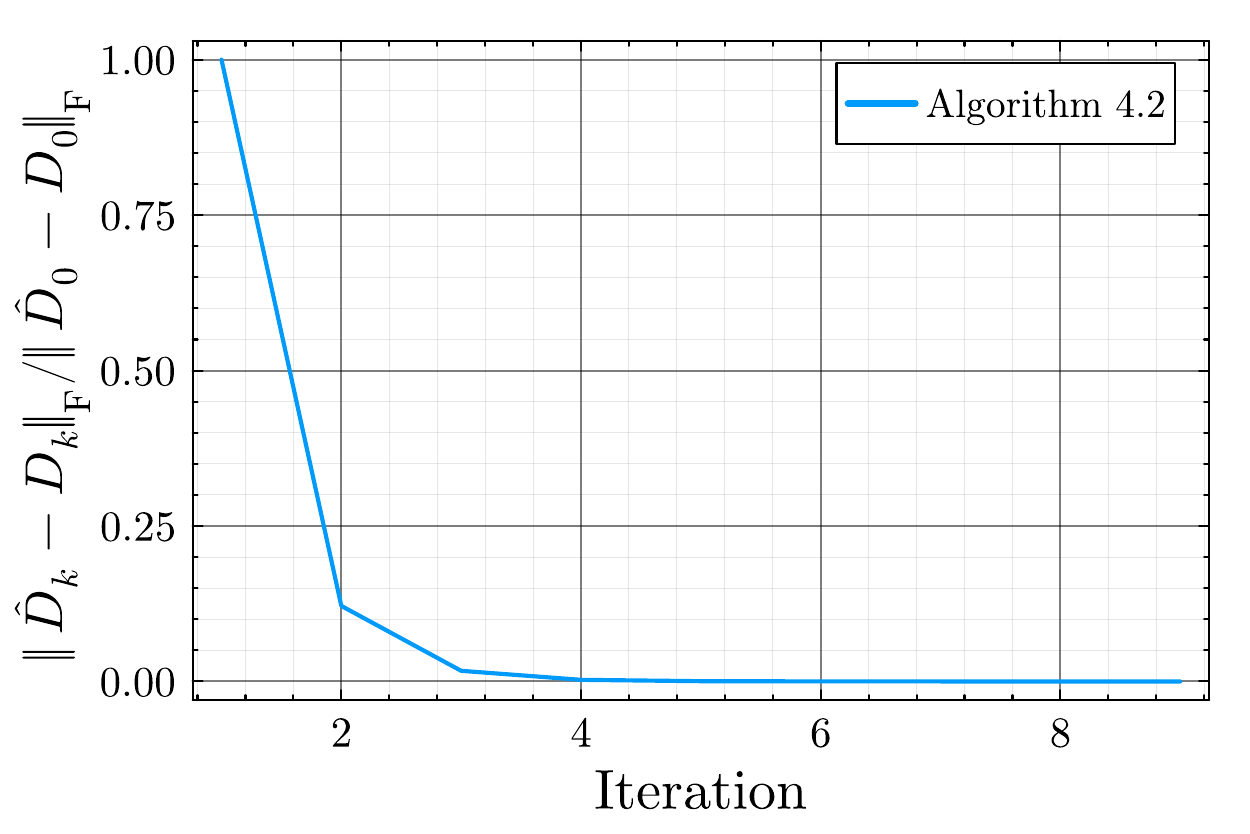}
     \includegraphics[width = 6cm]{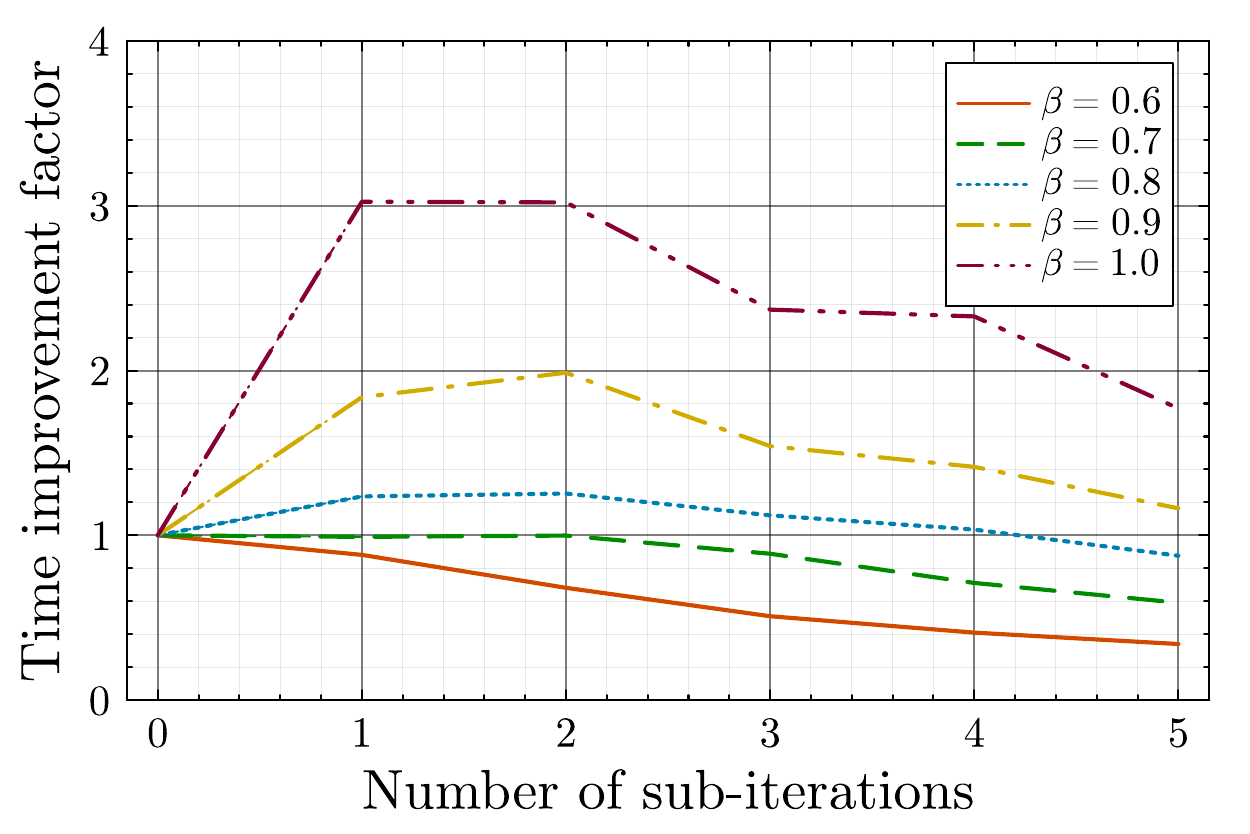}
     \vspace{-0.35cm}
     \caption{{\newrev On the left, the evolution of the residual $\|\m{D}_k-\widehat{\m{D}}_k\|_\mathrm{F}$ of~\cref{alg:subproblem} for a random matrix $V\in\mathrm{SO}(80)$ with $\|V-I_{80}\|_F=0.36\cdot2\sqrt{40}$. On the right, the \emph{time improvement factor} is the ratio between the running time of the fixed forward and the pseudo-backward method when $\beta$ and the number of sub-iterations vary. The stopping criterion is set to $(\|C_k\|_F+\|\widehat{A}_k-A_k\|_F)^*<10^{-12}$.  The experiment is performed on $\mathrm{St}(n=80,p=40)$ at distance $\|\m{U}-\widetilde{U}\|_\mathrm{F}=(0.36\pm0.01)\cdot2\sqrt{p}$}.}
     \label{fig:subproblemconvergence}\vspace{-0.3cm}
 \end{figure}
 \begin{figure}
     \centering
     \includegraphics[width =6.3cm]{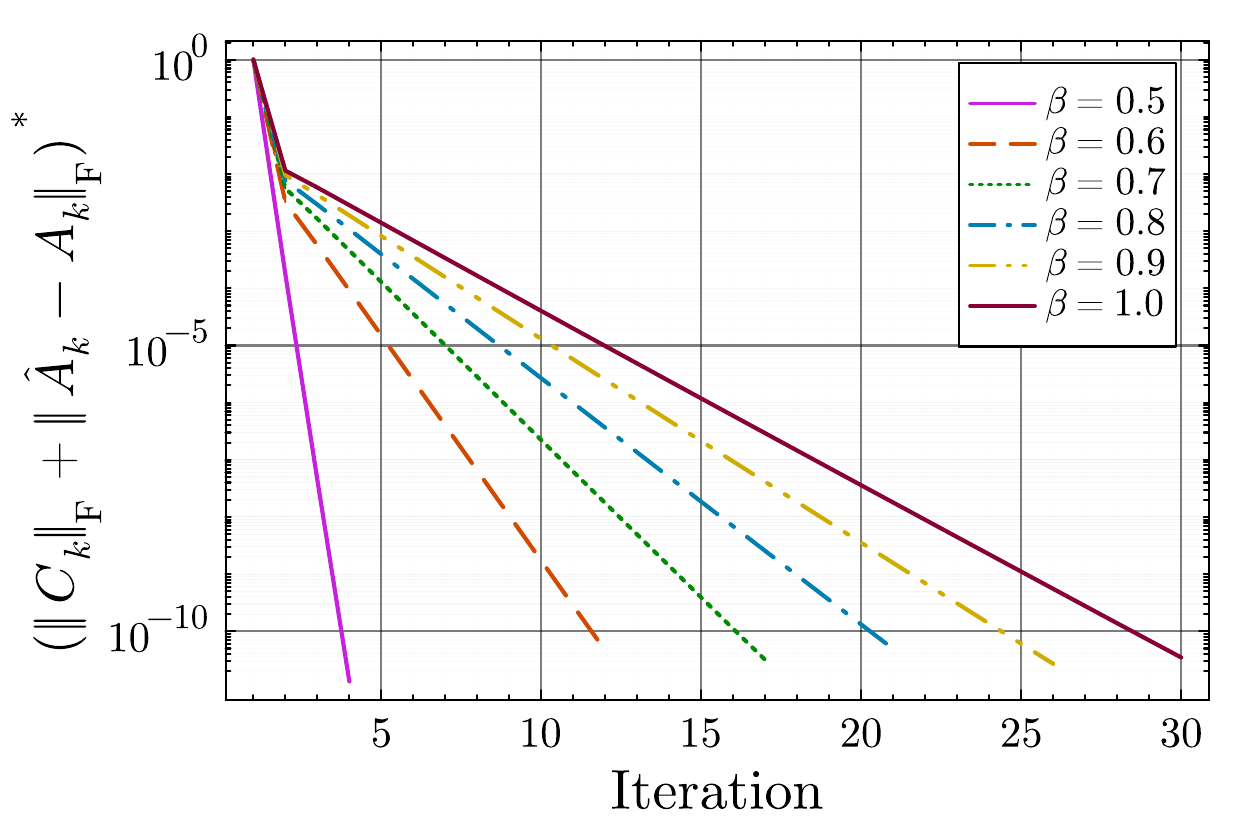}
     \includegraphics[width =6.3cm]{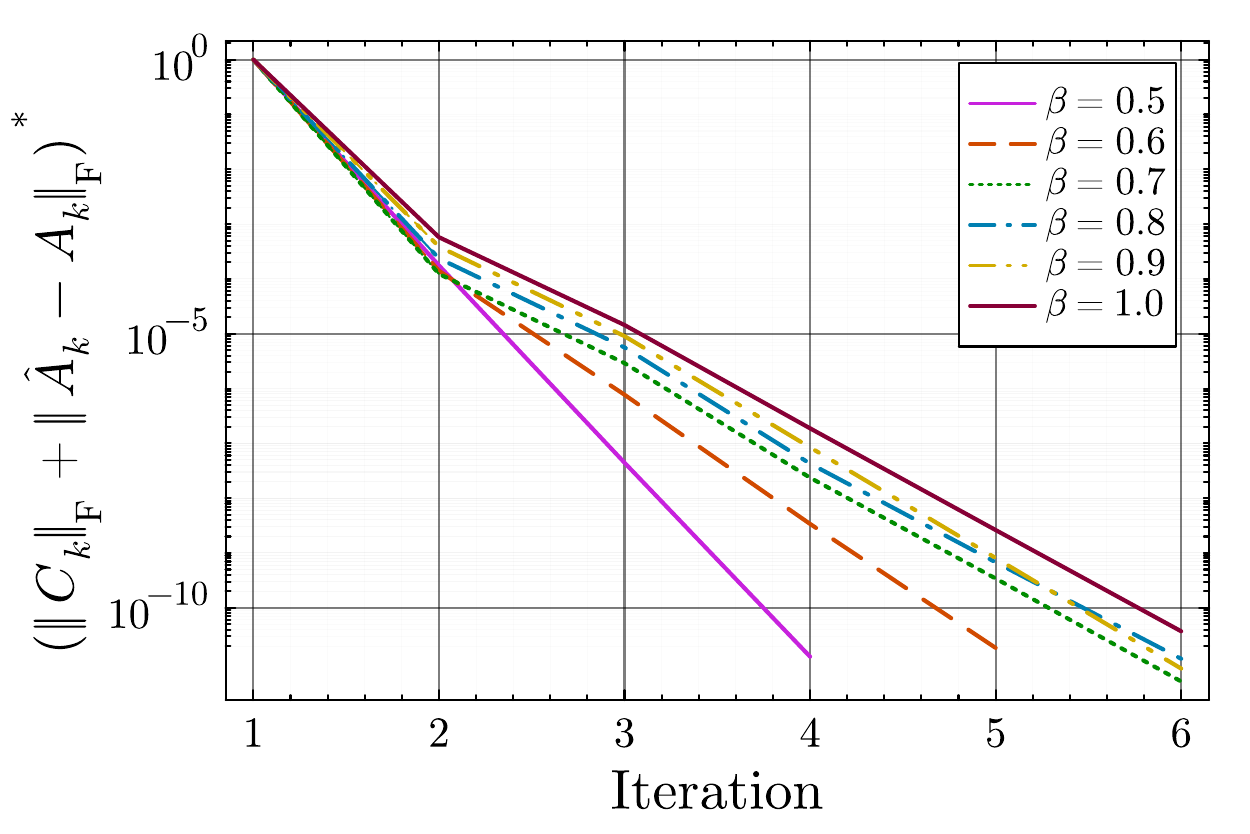}
     \includegraphics[width =6.3cm]{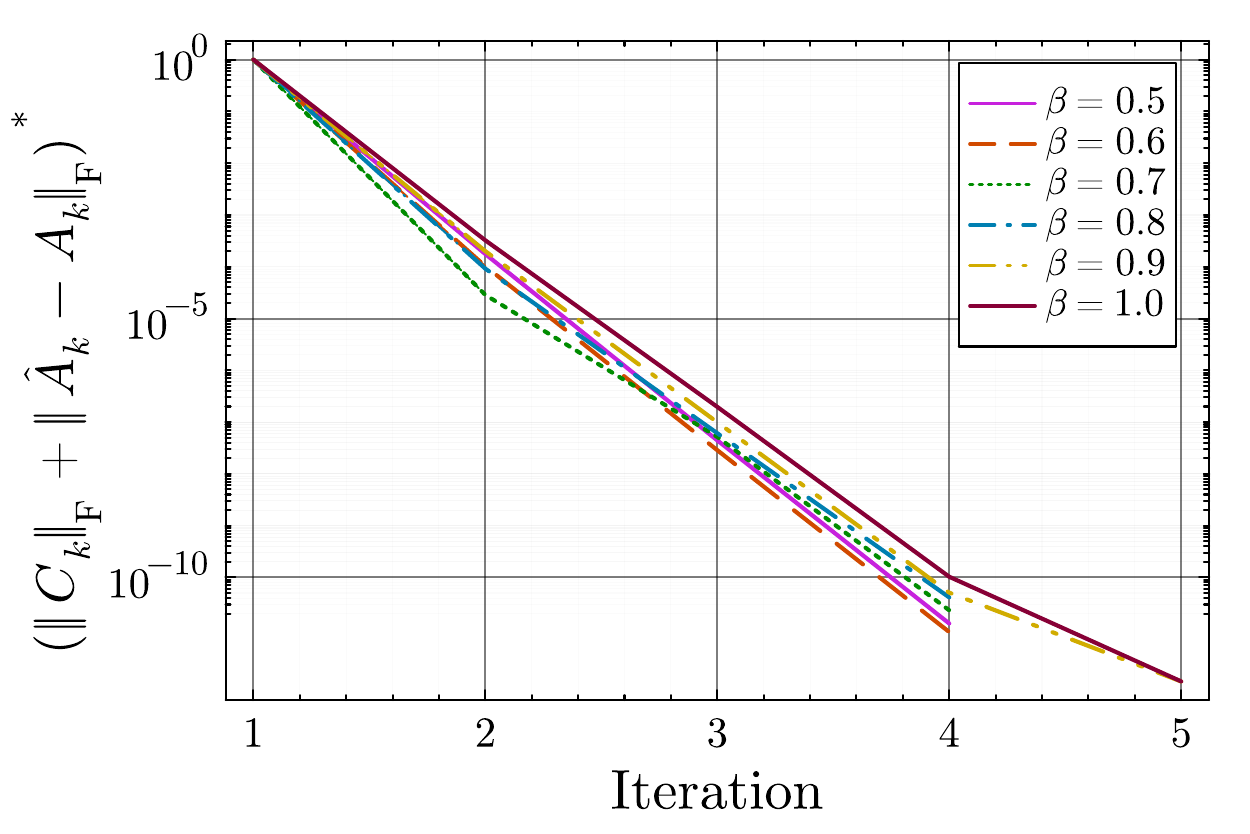}
     \vspace{-0.35cm}
     \caption{Evolution of the residuals $\|\m{C}_k\|_\mathrm{F}+\|\m{A}_k-\widehat{\m{A}}_k\|_\mathrm{F}$ for the fixed forward iteration (top left), pseudo-backward iteration with 1 sub-iterations (top right) and 2 sub-iterations (bottom) for a random experiment on $\mathrm{St}(n=80,p=30)$ with $\|\m{U}-\widetilde{U}\|_\mathrm{F} = 0.16\cdot2\sqrt{p} \approx 1.75$. The star `` $*$'' indicates that the residuals are normalized by the residual of the first iteration.
     For all plots, the curve for $\beta=0.5$ (solid purple line) is the same {\new since the all iterations reduce to the same algorithm for the canonical metric}.
     }
     \label{fig:subiterinfluence}
     \vspace{-0.5cm}
 \end{figure}
 
\section{Performance analysis}\label{sec:performance} We investigate the performance of~\cref{alg:generalizedStiefel} through two questions: how often and how fast does it converge? To answer the first one, we estimate a probabilistic convergence radius| that is, we find the distance $\|\m{U}-\widetilde{U}\|_\mathrm{F}$ such that~\cref{alg:generalizedStiefel} converges with probability close to 1, say 0.99. For the second question, we compare the running times of known methods, carefully implemented to extract the best performance for each of them. 

 \subsection{Probabilistic radius of convergence}\label{sec:logisticmodel}  From randomized numerical experiments, we fit a logistic model $m_\theta(x)\approx P(\text{Alg.~\ref{alg:generalizedStiefel} converged}\ | \ \|\widetilde{U}-\m{U}\|_\mathrm{F}=x)$ where $\theta$ is chosen as the maximum likelihood estimator. The details and motivation of this fitting are described in~\cref{app:logisticmodel}. The left plot in~\cref{fig:logisticregression} displays the fitting models on $\mathrm{St}(32,16)$. This fit provides a probabilistic radius of convergence $r$ by taking $r:=\max\{x\in\mathbb{R} \ |\ m_\theta(x)\geq0.99\}$, i.e., the probability of convergence is higher than $99\%$ if $\|\m{U}-\widetilde{U}\|_\mathrm{F}\leq r$. The right plot in~\cref{fig:logisticregression} shows that convergence is most challenging under the Euclidean metric ($\beta=1$). In this case, \cref{alg:generalizedStiefel} converges with high probability if the distance of the inputs is $\|U-\widetilde{U}\|_F<0.4\cdot2\sqrt{p}$.
\begin{figure}
    \centering
    \includegraphics[width=6.3cm]{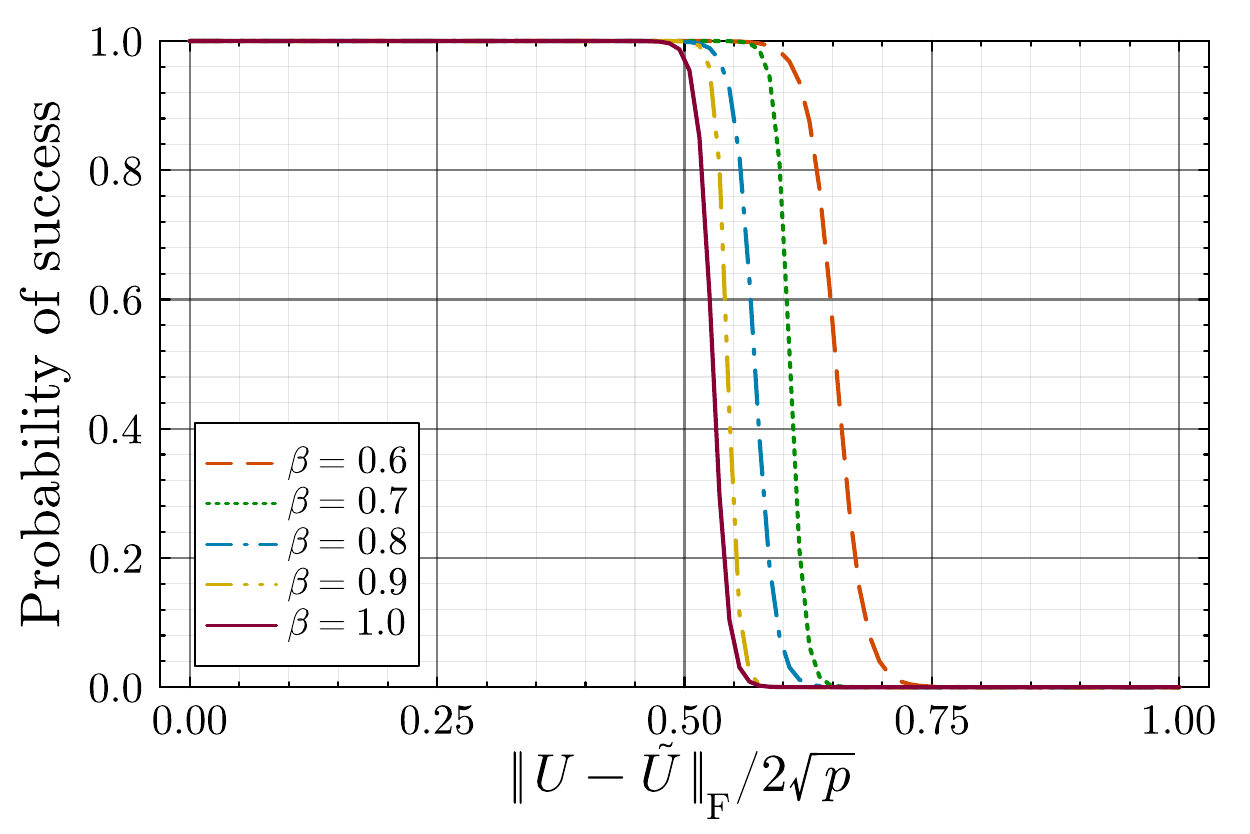}
    \includegraphics[width=6.3cm]{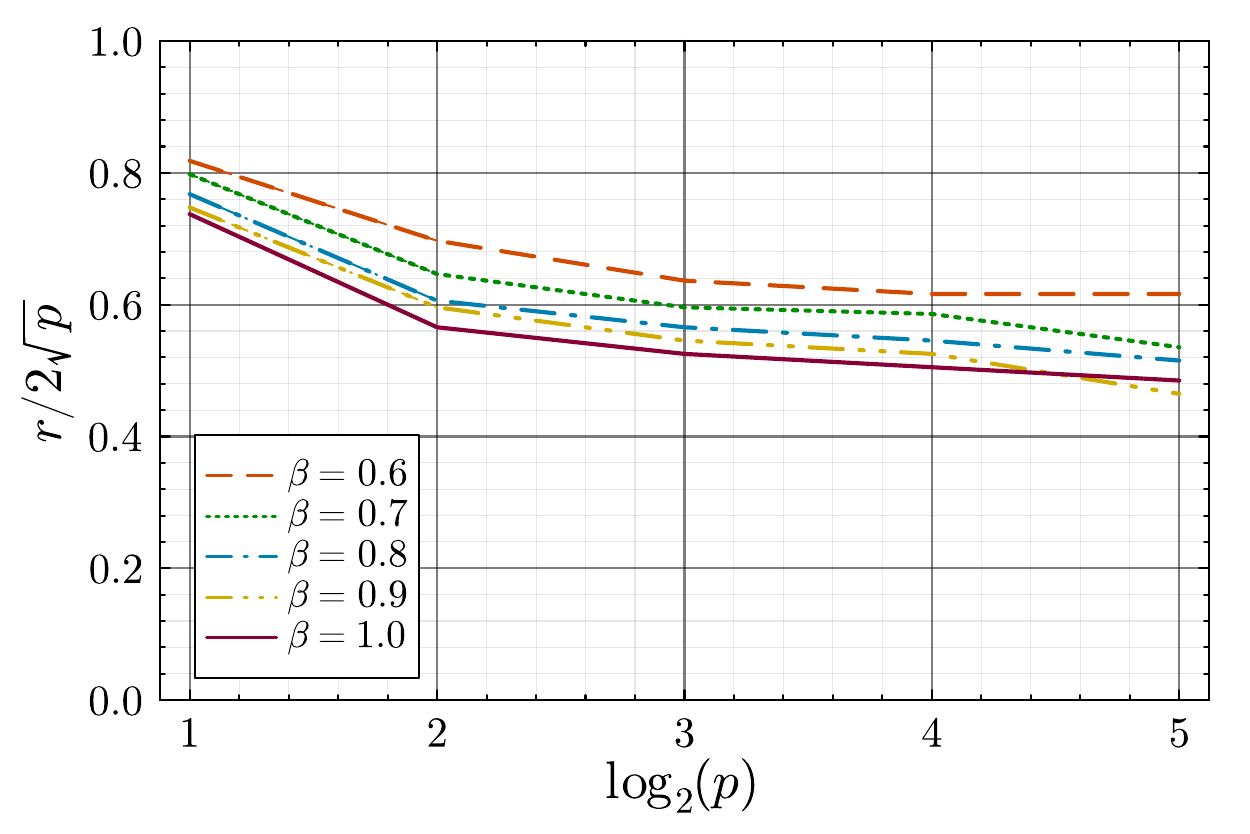}
    \vspace{-0.3cm}
    \caption{On the left,  the logistic regression models on  1000 random samples on $\mathrm{St}(32,16)$. The models are trained until $\|\nabla f(\theta)\|_F<10^{-8}$ (see~\cref{app:logisticmodel}). The logistic model estimates the probability of success of~\cref{alg:generalizedStiefel} (implemented with 2 pseudo-backward sub-iterations). Respectively for $\beta$ from $0.6$ to $1$, the R2 factors of the fitting are $\{0.85, 0.97,0.96,0.98,0.97\}$, confirming the goodness of fit. On the right, the evolution of the radius of convergence as $p$ varies on $\mathrm{St}(2p,p)$.}
    \label{fig:logisticregression}
    \vspace{-0.8cm}
\end{figure}
\begin{rem}
    When $\m{U},\widetilde{U}$ are not in the convergence radius of~\cref{alg:generalizedStiefel}, it is still possible to use it as a subroutine for a globally convergent method, e.g., the \emph{leapfrog method} of \cite{noakes_1998}. This method has already been considered on the Stiefel manifold \cite{sutti2023shooting,sutti2023leapfrog}. It was then combined with a shooting method. However, being a global method, the leapfrog method needs not to be used inside the convergence radius of~\cref{alg:generalizedStiefel}.
\end{rem} 
\subsection{Benchmark}\label{sec:benchmark} We benchmark the different versions of~\cref{alg:generalizedStiefel} with the $p$-shooting method \cite[Algorithm~2]{ZimmermannRalf22}. The benchmarking is performed on random matrices $U_i,\widetilde{U}_i\in \mathrm{St}(n,p)$ for $i=1,...,N$ generated at fixed Frobenius distance. In view of the homogeneity of $\mathrm{St}(n,p)$, the $U_i$'s can be chosen arbitrarily w.l.o.g, here by applying a Gram-Schmidt process on a random matrix. The $\widetilde{U}_i$'s are built within the convergence radius as follows. For $\m{A}_i\in \mathrm{Skew}(n)$ filled with i.i.d normally distributed entries $\sim\mathcal{N}(0,\delta)$, we build $\widetilde{U}_i = [U_i\ U_{i,\perp} ]\exp(\m{A}_i)\m{I}_{n\times p}$. Because of a ``Central Limit Theorem effect'', if $\delta$ is fixed and $n,p$ are large, the value of $\frac{\|U_i-\widetilde{U}_i\|_\mathrm{F}^2}{4p}\in[0,1]$ converges to an expected value. Our experiments compare
\begin{enumerate}
    \item \cref{alg:generalizedStiefel} using fixed forward iterations.
    \item \cref{alg:generalizedStiefel} using pseudo-backward iterations with 1 and 2 sub-iterations of~\cref{alg:subproblem}.
    \item \cref{alg:generalizedStiefel} using accelerated forward iterations.
    \item The $p$-shooting method \cite[Algorithm~2]{ZimmermannRalf22} with 3 and 5 discretization points, implemented as in \href{https://github.com/RalfZimmermannSDU/RiemannStiefelLog/tree/main/Stiefel_log_general_metric}{RiemannStiefelLog}.
\end{enumerate}
We consider two sizes of Stiefel manifolds, $\mathrm{St}(80,20)$  and $\mathrm{St}(100,50)$. We also consider two distances between the samples, expressed in percentage of the Frobenius diameter of $\mathrm{St}(n,p)$ ($=2\sqrt{p}$). We consider $15\%$ and $32\%$. Given a method $M$ that ran $T$ times with CPU times $t_{M,j}(U_i,\widetilde{U}_i)$ for $j=1,\ldots,T$, the running time $t_M$ of the method $M$ is computed as follows:\vspace{-0.3cm}
\begin{equation}\label{eq:timing}
    t_M = \frac{1}{N}\sum_{i=1}^N \min_{j=1,...,T}t_{M,j}(U_i,\widetilde{U}_i).\vspace{-0.3cm}
\end{equation}
The ``$\min$'' in~\eqref{eq:timing} is more standard than the average for benchmarking, as implemented for the \texttt{@btime} macro in \texttt{Julia} \cite{bezanson2017julia}.
For~\cref{tab:Table1}, we take $N=T=10$. The experiments from~\cref{tab:Table1} demonstrate the effectiveness of~\cref{alg:generalizedStiefel} compared to the shooting method for $\beta\in[0.3,1]$. The shooting method is only competitive with~\cref{alg:generalizedStiefel} when $\beta=1$, i.e., in the case of the Euclidean metric. In particular the accelerated forward iteration appears to be the best choice since it is very close in performance to the other types of iterations in case $\beta\leq\frac{1}{2}$ and outperforms them otherwise. 
    
\begin{table}
    \centering
    \footnotesize
    \begin{tabular}{||c||c|c|c|c|c|c|c|c||}
    \hline
    \multicolumn{9}{||c||}{$\mathrm{St}(80,20),\quad \|\m{U}-\widetilde{U}\|_\mathrm{F}=(0.15\pm0.01)\cdot2\sqrt{p}$}\\
        \hline
         $\beta$ &0.3&0.4&0.5&0.6&0.7&0.8&0.9&1  \\
         \hline
         Scale&\multicolumn{8}{c||}{$\cdot 10^{-2}$ sec.}\\
         \hline
         \hline
         Fixed Forward& 2.8 & 1.0 & \cellcolor{green!70!black!70}0.31 & 0.84 & 1.1 &  1.35 &  1.6 &  1.8\\
         \hline
         Pseudo-Backward 1 it.&1.5& 0.93&  0.31& \cellcolor{green!70!black!70}  0.76  &0.86 & 0.92 &  0.88 & 0.89\\
         \hline
         Pseudo-Backward 2 it.&1.35&  \cellcolor{green!70!black!70} 0.85 & 0.31 & 0.85&  0.84 & 0.85  & 0.86 & 0.93\\
         \hline
         Acc. Forward& \cellcolor{green!70!black!70}1.30 &  0.87  &0.31 & 0.79 &\cellcolor{green!70!black!70} 0.81  &\cellcolor{green!70!black!70}0.85   &\cellcolor{green!70!black!70}0.81 &\cellcolor{green!70!black!70} 0.80\\
         \hline
         p-shooting 3 pt.&2.3 & 2.1  &1.8  &1.8 & 1.5  &1.3 & 15(\xmark) & 0.92\\ 
         \hline
         p-shooting 5 pt.&2.9&  2.7 & 2.2 & 2.2 & 2.0 & 1.8   &1.4 & 1.2\\ 
         \hline
         \hline
         \multicolumn{9}{||c||}{$\mathrm{St}(80,20),\quad \|\m{U}-\widetilde{U}\|_\mathrm{F}=(0.32\pm0.01)\cdot2\sqrt{p}$}\\
        \hline
         $\beta$ &0.3&0.4&0.5&0.6&0.7&0.8&0.9&1  \\
         \hline
         Scale&\multicolumn{8}{c||}{$\cdot 10^{-2}$ sec.}\\
         \hline
         \hline
         Fixed Forward& 3.3& \cellcolor{green!70!black!70} 1.2&  0.44&  \cellcolor{green!70!black!70}0.95&  1.2 & 1.6  &1.9&  2.2\\ 
         \hline
         Pseudo-Backward 1 it.&2.7&  1.3&  0.45&  1.1&  \cellcolor{green!70!black!70}1.2&  1.3 & 1.5 & 1.6\\ 
         \hline
         Pseudo-Backward 2 it.&\cellcolor{green!70!black!70}2.6& 1.4 & \cellcolor{green!70!black!70}0.42&  1.4 & 1.4 & 1.3 & 1.5  &1.6\\ 
         \hline
         Acc. Forward& 2.7& 1.3 &0.45&  1.1&  1.3&\cellcolor{green!70!black!70}  1.3& \cellcolor{green!70!black!70}1.3& \cellcolor{green!70!black!70} 1.3\\ 
         \hline
         p-shooting 3 pt.&4.3&  3.8 & 2.8 & 2.8 & 2.4 & 17(\xmark) & 1.7 & 1.6\\ 
         \hline
         p-shooting 5 pt.&5.3&  4.6&  3.6 & 3.4 & 2.9 & 2.6 & 2.3 &1.8\\ 
         \hline
         \hline
         \multicolumn{9}{||c||}{$\mathrm{St}(100,50),\quad \|\m{U}-\widetilde{U}\|_\mathrm{F}=(0.32\pm0.01)\cdot2\sqrt{p}$}\\
        \hline
         $\beta$ &0.3&0.4&0.5&0.6&0.7&0.8&0.9&1  \\
         \hline
         Scale&\multicolumn{8}{c||}{$\cdot 10^{-1}$ sec.}\\
         \hline
         \hline
         Fixed Forward&5.2 \cellcolor{green!70!black!70}& 1.8 & \cellcolor{green!70!black!70} 0.71&  1.5&  2.1 &  3.0 & 4.3&   7.4\\ 
         \hline
         Pseudo-Backward 1 it.&5.6 & 1.6 \cellcolor{green!70!black!70} & 0.72 & 1.4 & 1.7&   1.9 & 2.0  & 2.1\\ 
         \hline
         Pseudo-Backward 2 it.&5.8 & 1.8  & 0.71 & 1.5 & 1.7&  1.7 & 1.9  & 1.9\\ 
         \hline
         Acc. Forward&8.1 & 1.9  & 0.72  &\cellcolor{green!70!black!70}1.4& \cellcolor{green!70!black!70} 1.7 &  \cellcolor{green!70!black!70}1.7& \cellcolor{green!70!black!70} 1.6 &  1.7\\ 
         \hline
         p-shooting 3 pt.&38(\xmark)& 5.1 & 12(\xmark) & 3.1 & 2.5 & 33(\xmark) & 1.6&  22(\xmark)\\ 
         \hline
         p-shooting 5 pt.&18(\xmark)&  5.5 &  3.5 & 3.3&  2.6 &  2.1 & 1.9 & \cellcolor{green!70!black!70} 1.6\\ 
         \hline
    \end{tabular}
    \caption{Benchmark of~\cref{alg:generalizedStiefel} in its different versions (fixed forward, pseudo-backward with 1 or 2 sub-iterations and accelerated forward). It is compared to the $p$-shooting method \cite[Algorithm~2]{ZimmermannRalf22} with $3$ and $5$ points (starting and ending point included). With only $3$ points, the appearance of failures (\xmark) indicates that more points must be considered for robustness. {\new For each value of $\beta$, the cell of the best performing method is highlighted.}}
    \label{tab:Table1}
    \vspace{-0.9cm}
\end{table}

\section{Conclusion}
In this paper, we have proposed an alternative to the shooting method to compute a geodesic between any two points on the Stiefel manifold. The method generalizes the approach of \cite{Zimmermann17} to the complete family of metrics introduced in \cite{Huper2021}. Our analysis included theoretical guarantees and numerical experiments. Future work may include the design of a more robust initialisation, increasing significantly the current restrictive radius of convergence. A more detailed understanding of the effectiveness of the accelerated forward iteration should also be carried out.
\section{Acknowledgments} The authors would like to thank Pierre-Antoine Absil for his invaluable guidance during the preparation of this manuscript and anonymous referees for their constructive comments and valuable suggestions, which have significantly improved the quality of this work. 

\newpage
\appendix
\section{Rank of logarithm}\label{app:rankdeficiency} Given $U,\widetilde{U}\in\mathrm{St}(n,p)$, we investigate the cases where $\Delta\in\mathrm{Log}_{\beta,\m{U}}(\widetilde{U})$ and
\begin{equation}\label{eq:rank}
        \mathrm{rank}((\m{I} - \m{U}\m{U}^T)\Delta)>\mathrm{rank}((\m{I} - \m{U}\m{U}^T)\widetilde{U}). 
    \end{equation}
An example where~\eqref{eq:rank} holds is for antipodal points ($\widetilde{U}=-U$) on the hypersphere $\mathrm{St}(n,1)$. We show that~\eqref{eq:rank} can only happen when $\widetilde{U}$ belongs to the cut locus of $U$, which is a zero-measure subset of $\mathrm{St}(n,p)$ \cite[Lemma~4.4]{sakai1996riemannian}.
\begin{proposition}
    Assume $U, \widetilde{U}\in\mathrm{St}(n,p)$ and~\eqref{eq:rank} holds. Then $\widetilde{U}$ belongs to the cut locus of $U$.
\end{proposition}
\begin{proof}
Assume $\Delta\in \mathrm{Log}_{\beta,U}(\widetilde{U})$. By definition, we can always write $\Delta = U A +  \begin{bmatrix}
    Q_1&Q_2
\end{bmatrix}\begin{bmatrix}
    B_1\\
    B_2
\end{bmatrix}$ where $\mathrm{col}([U\  Q_1])=\mathrm{col}([U\  \widetilde{U}])$ and $Q_2^TU=Q_2^T\widetilde{U} = 0$. By the assumption~\eqref{eq:rank}, we have $Q_2,B_2\neq 0$. Let $Q_2\in\mathrm{St}(n,q)$ with $q\geq1$. Then, we have
\begin{align*}
    \widetilde{U}&=\mathrm{Exp}_{\beta,U}(\Delta)\\
    \iff\begin{bmatrix}
        U^T\widetilde{U}\\
        Q_1^T\widetilde{U}\\
        0
    \end{bmatrix} &= \exp\begin{bmatrix}
            2\beta A&\begin{matrix}
                -\m{B}_1^T&
                -\m{B}_2^T
            \end{matrix}\\
            \begin{matrix}
                \m{B}_1\\
                \m{B}_2
            \end{matrix}&\m{0}_{n-p}
        \end{bmatrix}I_{n\times p}\exp(\tau A)\\
        \iff \begin{bmatrix}
        U^T\widetilde{U}\\
        Q_1^T\widetilde{U}\\
        0
    \end{bmatrix} &= \exp\begin{bmatrix}
            2\beta A&\begin{matrix}
                -\m{B}_1^T&
                -(R\m{B}_2)^T
            \end{matrix}\\
            \begin{matrix}
                \m{B}_1\\
                R\m{B}_2
            \end{matrix}&\m{0}_{n-p}
        \end{bmatrix}I_{n\times p}\exp(\tau A)\text{ for all } R\in \mathrm{O}(q).
\end{align*}
Therefore, for all $R\in \mathrm{O}(q)$ and $\widetilde{\Delta} := U A +  \begin{bmatrix}
    Q_1&Q_2
\end{bmatrix}\begin{bmatrix}
    B_1\\
    R B_2
\end{bmatrix}$, $\mathrm{Exp}_{\beta,U}(\widetilde{\Delta})=\widetilde{U}$ and $\|\widetilde{\Delta}\|_\beta=\|\Delta\|_\beta$ yield $\widetilde{\Delta}\in\mathrm{Log}_{\beta, U}(\widetilde{U})$. In conclusion, there is more than one minimal geodesic from $U$ to $\widetilde{U}$ and $\widetilde{U}$ is on the cut locus of $U$ by \cite[Proposition~4.1]{sakai1996riemannian}.
\end{proof}
\section{Initialization of $V_0$}\label{app:initialization} Both~\cref{alg:Zimmermannsalgorithm} and~\cref{alg:generalizedStiefel} ask for an initialization $V_0=\left[\begin{smallmatrix}
    M&O_0\\
    N&P_0
\end{smallmatrix}\right]\in\mathrm{SO}(2p)$. $M=U^T\widetilde{U}$ is fixed by the problem. However, $N$, $O_0$, $P_0$ and $Q$ have degrees of freedom to obtain the best performance. Let us recall their meaning. From~\cref{thm:geodesics}, we know that for $n\leq2p$, the geodesic can be considered as a curve evolving in $\mathrm{St}(n, 2p)$, from which we only retain the $p$ first columns. In $\mathrm{St}(n, 2p)$, the geodesic starts at $[U\ Q]$ and ends at $[\widetilde{U}\ \widetilde{Q}]$. The underlying goal of~\cref{alg:Zimmermannsalgorithm} and~\cref{alg:generalizedStiefel} is to obtain $Q,\widetilde{Q}\in\mathrm{St}(n,p)$ minimizing the distance from $[U\ Q]$ to $[\widetilde{U}\ \widetilde{Q}]$. From this framework, $V_0$ can be interpreted as 
\begin{equation}
    V_0 := \begin{bmatrix}
        U^T\widetilde{U}& U^T\widetilde{Q}\\
        Q^T\widetilde{U}&Q^T\widetilde{Q}
    \end{bmatrix}\in\mathrm{SO}(2p).
\end{equation}
A heuristic initialization is finding $Q,\widetilde{Q}$ minimizing the geodesic distance $d(I, V_0)$ in $\mathrm{SO}(2p)$. It is known that $d(I, V_0)\propto \|\log(V_0)\|_F$ (see, e.g., \cite{EdelmanArias98}). Since $U, \widetilde{U}$ are fixed, an easier problem is to solve $\min_{Q,\widetilde{Q}} \|I-Q^T\widetilde{Q}\|_F$. Since the column spaces of $Q,\widetilde{Q}$ are known, this second problem is an Orthogonal Procrustes problem, solved by the SVD. It admits an infinite set of solutions, given by all orthogonal similarity transformations of $Q^T\widetilde{Q}$. Let us build these solutions. If we take $\widehat{Q}\widehat{N} = (I-U U^T)\widetilde{U}$ (with $\widehat{Q}\in\mathrm{St}(n,p)$ and $\widehat{Q}^TU=0$), we can start with any $\left[\begin{smallmatrix}
    M&\widehat{O}_0\\
    \widehat{N}&\widehat{P}_0
\end{smallmatrix}\right]\in\mathrm{SO}(2p)$. The idea from \cite{Zimmermann17} is to compute $\widehat{P}_0 = R\Sigma\widetilde{R}^T$, a singular value decomposition.
\begin{itemize}
    \item Method from \cite{Zimmermann17}: take $Q:=\widehat{Q}$, $N :=\widehat{N}$, $O_0 := \widehat{O}_0\widetilde{R} R^T$, $P_0 = R\Sigma R^T$.
    \item Our method: take $Q:=\widehat{Q}R$, $N = R^T\widehat{N}$, $O_0 = \widehat{O}_0\widetilde{R}$, $P_0 = \Sigma$.
\end{itemize}
Our method corresponds to a similarity transformation of the method from \cite{Zimmermann17} by $\left[\begin{smallmatrix}
    I&0\\
    0&R
\end{smallmatrix}\right]$. It offers a diagonal matrix $P_0$.
\section{Intermediate results for~\cref{sec:convergencebackward}}\label{app:appendixA} This appendix gathers intermediate results in the proof of~\cref{thm:convergence}. It allows to keep a clearer narration in~\cref{sec:convergencebackward}. 
First, \cref{lem:normofX} allows to relate the norm of two matrices appearing in~\cref{thm:convergence}.
\begin{lem}\label{lem:normofX}
    Let $\m{L}_k:=\begin{bmatrix}
        \m{A}_k&-\m{B}_k^T\\
        \m{B}_k&\m{C}_k
    \end{bmatrix}$ and $\m{X}_k:=\begin{bmatrix}
        2\beta\m{A}_k&-\m{B}_k^T\\
        \m{B}_k&\m{C}_k
    \end{bmatrix}$ 
    produced by~\cref{alg:generalizedStiefel}, then $\|\m{X}_k\|_2 \leq (1+ |2\beta-1|)\|\m{L}_k\|_2$ and $\|\m{L}_k\|_2 \leq  \left(1+\frac{|2\beta-1|}{2\beta}\right)\|\m{X}_k\|_2$.
\end{lem}
\begin{proof}
Simply notice that
    \begin{equation*}
        X_k = L_k +
        \begin{pmatrix}
        (2\beta-1) A_k & 0 \\
        0 & 0
        \end{pmatrix}\Rightarrow
        \|X_k\|_2 \leq \|L_k\|_2 + |2\beta-1|\|L_k\|_2,
    \end{equation*}
    and
    \begin{equation*}
        L_k = X_k + \begin{pmatrix}
        (1-2\beta) A_k & 0 \\
        0 & 0
        \end{pmatrix} \Rightarrow \|L_k\|_2\leq \|X_k\|_2+\frac{|2\beta-1|}{2\beta}\|X_k\|_2.
    \end{equation*}
    This concludes the proof.
\end{proof}
Then, \cref{lem:bound_hot_theta} allows to bound the higher order terms from \eqref{eq:expansion_of_theta} in terms of the norm of $\|A_k-A_{k+1}\|_2$.
\begin{lem}\label{lem:bound_hot_theta}
    In  \eqref{eq:expansion_of_theta}, we can bound the higher order terms by  $\|\mathrm{H.O.T}_\Theta(4)\|_2\leq|\tau|(|\tau|\delta)^2\log\left(\frac{1}{1-2|\tau|\delta}\right) \|A_k-A_{k+1}\|_2$.
\end{lem}
\begin{proof}
    We know from the commutator version of the Goldberg series \cite{NewmanThompson87} (or BCH formula) that
\begin{equation*}
    \|\mathrm{H.O.T}_\Theta(4)\|_2\leq\sum_{\  l=4}^\infty\sum_{w_l}\frac{|g_{w_l}|}{l}\| [w_l(\tau A_k,\tau A_{k+1})] \|_2,
\end{equation*}
where $w_l(\tau A_k,\tau A_{k+1})$ is a word of length $l$ in the alphabet $\{\tau A_k,\tau A_{k+1}\}$ and the notation $[w_l(\tau A_k, \tau A_{k+1})]$, $[w_l]$ for short, is the extended commutator defined on this word \cite{NewmanThompson87}. $g_{w_l}$~is the Goldberg coefficient associated to $w_l$ \cite{Goldberg1956}. Since $\|[A,B]\|_2\leq 2\|A\|_2\|B\|_2$ and $\|[A,B]\|_2\leq 2\|A\|_2\|B-A\|_2$, it follows by recurrence that $\|[w_l]\|_2\leq 2^{l-1}\delta^{l-1}\tau^l\|A_k-A_{k+1}\|_2$.  Moreover, for words of length $l$, we have $\sum_{w_l}\frac{|g_{w_l}|}{l}\leq\frac{2}{l}$ \cite{THOMPSON1989}. \cite[Lem.~A.1]{Zimmermann17} even decreased this bound to $\sum_{w_l}\frac{|g_{w_l}|}{l}\leq\frac{1}{l}$. It follows that 
\begin{align*}
    \|\mathrm{H.O.T}_\Theta(4)\|_2&\leq\sum_{l=4}^\infty \sum_{w_l}\frac{|g_{w_l}|}{l} 2^{l-1}\delta^{l-1}\tau^l \|A_k-A_{k+1}\|_2 \\
    &\leq |\tau|(|\tau|\delta)^2 \|A_k-A_{k+1}\|_2\sum_{l=1}^\infty (2|\tau|\delta)^l \left(\sum_{ w_{l+3}}\frac{|g_{w_{l+3}}|}{l+3}\right)\\
    \end{align*}
   \begin{align*}
     \|\mathrm{H.O.T}_\Theta(4)\|_2&\leq |\tau|(|\tau|\delta)^2 \|A_k-A_{k+1}\|_2\sum_{l=1}^\infty \frac{(2|\tau|\delta)^l}{l}\\
    &\leq  |\tau|(|\tau|\delta)^2\log\left(\frac{1}{1-2|\tau|\delta}\right) \|A_k-A_{k+1}\|_2,
\end{align*}
where $2|\tau|\delta<1$ stands by~\cref{cond:delta}.
\end{proof}
\cref{lem:partialproofthmconvergence} shows that it is indeed valid to go
from~\eqref{eq:thetafirstbound} to~\eqref{eq:equationfornormdeltaA} in the proof of~\cref{thm:convergence}.
\begin{lem}\label{lem:partialproofthmconvergence}
   Equation \eqref{eq:equationfornormdeltaA} follows from~\eqref{eq:thetafirstbound}.
\end{lem}
\begin{proof} 
The top-left  $p\times p$ block of the BCH series expansion of $\m{A}_{k+1}$ based on the fundamental equation~\eqref{eq:fundamentalequation} yields

\begin{align}
    \label{eq:equationfordeltaA}
    2\beta(\m{A}_{k+1}-\m{A}_k)&=\m{\Theta}_k+\beta [\m{A}_k,\m{\Theta}_k]\\
    \nonumber
    &+\frac{1}{12}\Big(4\beta^2[\m{A}_k, [\m{A}_k,\m{\Theta}_k]]-\m{B}_k^T\m{B}_k\m{\Theta}_k-\m{\Theta}_k\m{B}_k^T\m{B}_k+2\m{B}_k^T\m{\Gamma}_k\m{B}_k\\
    \nonumber
    &-2\beta[\m{\Theta}_k,[\m{A}_k,\m{\Theta}_k]]\Big)-\frac{1}{24}\Big(2[\m{\Theta}_k,\m{B}_k^T\m{\Gamma}_k\m{B}_k]\Big){\new +\mathcal{O}(\|\Theta_k\|_2^2)}\\
    \nonumber
    &+\text{H.O.T}_{\m{A}}(5).
\end{align} 
Equation \eqref{eq:equationfordeltaA} features many terms that we tackle one by one. All the terms can be bounded by leveraging~\eqref{eq:thetafirstbound} and $\|\Gamma_k\|_2\leq \frac{6}{6-\delta^2}\|C_k\|_2\leq\frac{6\delta}{6-\delta^2}$ \cite{ZimmermannRalf22}:
\begin{itemize}
        \item[\ding{226}] $\|[\m{A}_k,\m{\Theta}_k]\|_2\leq 2 \|\m{A}_k \|_2 \|\m{\Theta}_k\|_2\leq 2\delta\eta\|\m{A}_k-\m{A}_{k+1}\|_2.$
        \item[\ding{226}] $\|[\m{A}_k, [\m{A}_k,\m{\Theta}_k]]\|_2\leq 2\|\m{A}_k \|_2 \|[\m{A}_k,\m{\Theta}_k]\|_2\leq 4\delta^2\eta \|\m{A}_k-\m{A}_{k+1}\|_2.$
        \item[\ding{226}] $\|\m{B}_k^T\m{B}_k\m{\Theta}_k\|_2 \leq \delta^2 \eta\|\m{A}_k-\m{A}_{k+1}\|_2.$
       \item[\ding{226}] $\|\m{B}_k^T\m{\Gamma}_k\m{B}_k\|_2\leq \frac{6\delta^2}{6-\delta^2}\|\m{C}_k\|_2.$
       \item[\ding{226}] $\|[\m{\Theta}_k,[\m{A}_k,\m{\Theta}_k]]\| \in \mathcal{O}(\|\m{A}_k-\m{A}_{k+1}\|_2^2).$
       \item[\ding{226}] $\|[\m{\Theta}_k,\m{B}_k^T\m{\Gamma}_k\m{B}_k]\|_2\leq 2 \delta^2 \|\m{\Gamma}_k\|_2\|\m{\Theta}_k\|_2\leq \frac{12\delta^3}{6-\delta^2}\eta\|\m{A}_k-\m{A}_{k+1}\|_2.$
       \item[\ding{226}]{\new  $\mathcal{O}(\|\Theta_k\|_2^2)\in\mathcal{O}(\|\m{A}_k-\m{A}_{k+1}\|_2^2)$}
    \end{itemize}
    Inserting all these terms in~\eqref{eq:equationfordeltaA} yields~\eqref{eq:equationfornormdeltaA}.
\end{proof}

\section{An alternative algorithm to solve~\cref{prob:subproblem}}\label{app:shootingsubproblem}
\cref{alg:subproblemshooting} proposes a shooting method on $\mathrm{SO}(2p)$ to solve~\cref{prob:subproblem}, inspired from \cite[Algorithm~1]{BrynerDarshan17}. In~\cref{prob:subproblem}, notice that the initial shooting direction is $\m{M}:=\Dot{\gamma}(0)=\left[\begin{smallmatrix}
    \m{D}&-\m{B}^T\\
    \m{B}&\m{C}
\end{smallmatrix}\right]$. Starting from an initial guess for $\m{M}_0=\Dot{\gamma}_0(0)$, $\m{M}_k$ is updated  using an approximate parallel transport of the error vector $\D_k:=\m{V}-\gamma_{k}(1)$ along the curve $\gamma_{k}$. We follow the method of \cite[Algorithm~1]{BrynerDarshan17} to approximate this parallel transport of $\D_k$ to the tangent space of $\gamma_{k}(0)=\m{I}_{2p}$,  written $T_{{\m{I}_{2p}}}\mathrm{SO}(2p)$: $\D_k$ is sequentially projected on $T_{\gamma_{k}(t)}\mathrm{SO}(2p)$ for $t\in[t_m,t_{m-1},...,t_1]$ with $t_m$=1 and $t_1=0$. The pseudo-code of the method is provided in~\cref{alg:subproblemshooting}.
\begin{algorithm}
    \caption{The subproblem's shooting algorithm}
    \label{alg:subproblemshooting}
    \begin{algorithmic}[1]
        \STATE \textbf{INPUT:} Given $\m{V}\in\mathrm{SO}(2p)$, $\beta>0$, $\varepsilon>0$ and $[t_1,t_2,...,t_m]$ with $ t_1=0$ and $t_m=1$.
        \STATE Initialize $\m{M}_0:=\begin{bmatrix}
                    \m{D}_0&-\m{B}_0^T\\
                    \m{B}_0&\m{C}_0
                \end{bmatrix}$ and $k=0$.
        \WHILE{$\nu>\epsilon$}
            \FOR{$j=m,m-1,...,1$}
                \STATE   $\m{V}^s\leftarrow\exp\left(t_j\begin{bmatrix}
                    2\beta \m{D}_k&-\m{B}_k^T\\
                    \m{B}_k&\m{C}_k
                \end{bmatrix}\right)\exp\left(t_j\begin{bmatrix}
                    (1-2\beta) \m{D}_k&\m{0}\\
                    \m{0}&\m{0}
                \end{bmatrix}\right)$
                \IF{$j==m$}
                    \STATE $\m{W}\leftarrow\m{V}-\m{V}^s$
                    \STATE $\nu \leftarrow \|\m{\Delta}\|_\mathrm{F}$
                \ENDIF
                \STATE $\m{M}^s\leftarrow \mathrm{skew}((\m{V}^s)^T\m{W})$\hspace{4cm}\text{\#Project $\m{W}$ on $T_{\m{V}^s}\mathrm{SO}(2p)$.}
                \STATE $\m{W}\leftarrow \m{V}^s\m{M}^s\cdot\frac{\nu}{\|\m{M}^s\|_\mathrm{F}}$
            \ENDFOR
            \STATE $\m{M}_k\leftarrow\m{M}_k + \m{M}^s$, $k= k+1$.
        \ENDWHILE
        \RETURN $\m{D}_k$
    \end{algorithmic}
\end{algorithm}

 \section{Logistic model fitting}\label{app:logisticmodel}
 We obtain a probabilistic radius of convergence of~\cref{alg:generalizedStiefel} from numerical experiments. First, define a function $\mathcal{X}_\beta:\mathrm{St}(n,p)\times\mathrm{St}(n,p)\mapsto\{0,1\}$ where
\begin{equation*}
    \mathcal{X}_\beta(\m{U},\widetilde{U})=\begin{cases}
        1 \text{ if~\cref{alg:generalizedStiefel} converged with $(\m{U},\widetilde{U}, \beta)$ as input.}\\
        0 \text{ otherwise.}
    \end{cases}
\end{equation*}
Given $\beta>0$ and $N$ samples $\{U_i,\widetilde{U}_i\}_{i\in\{1...,N\}}$ drawn from a continuous distribution on $\mathrm{St}(n,p)$, we build a data set of pairs 
\begin{equation*}
    \{x_i,y_i\}_{i\in\{1...,N\}}:=\{\|U_i-\widetilde{U}_i\|_\mathrm{F},\ \mathcal{X}_\beta(U_i,\widetilde{U}_i)\}_{i\in\{1...,N\}}.
\end{equation*} 
We expect~\cref{alg:generalizedStiefel} to converge when $\|U_i-\widetilde{U}_i\|_\mathrm{F}$ is small and failures to appear  when $\|U_i-\widetilde{U}_i\|_\mathrm{F}$ gets larger. This framework is natural to fit a logistic regression model $m_\theta:\mathbb{R}\mapsto(0,1)$, where $\theta:=[\theta_0,\ \theta_1]\in \mathbb{R}^2$ is the fitting parameter. The logistic model predicts the probability of convergence of~\cref{alg:generalizedStiefel}. It is given by
\begin{equation*}
    m_\theta(x)=\frac{1}{1+\exp(\theta_0+\theta_1x)}.
\end{equation*}
The estimator $\theta$ is chosen to maximize the likelihood of the data set, i.e.,
\begin{equation*}
    \theta:=\mathrm{arg}\max_{\theta\in\mathbb{R}^2} \prod_{i=1}^N m_\theta(x_i)^{y_i}(1-m_\theta(x_i))^{(1-{y_i})}.
\end{equation*}
{\new For reference, see, e.g., \cite{Maalouf11}.} It is well-known that it is easier and equivalent to obtain the log-likelihood estimator
\begin{equation}\label{eq:loglikelihood}
    \theta:=\mathrm{arg}\max_{\theta\in\mathbb{R}^2} f(\theta):=\mathrm{arg}\max_{\theta\in\mathbb{R}^2}\sum_{i=1}^N \log(m_\theta(x_i))y_i+\log(1-m_\theta(x_i))(1-{y_i}).
\end{equation}
Equation \eqref{eq:loglikelihood} is a Lipschitz-smooth convex optimization problem and we solve it using the accelerated gradient method \cite{Nesterov1983AMF} with stopping criterion $\|\nabla f(\theta)\|_\mathrm{F}<10^{-8}$. We considered $N=1000$ in our experiments. The goodness of fit is confirmed by the high coefficients of determination $R^2$ provided in the caption of~\cref{fig:logisticregression}.

\end{document}